\documentclass[12pt, reqno]{amsart}
\usepackage{amsmath, amsthm, amscd, amsfonts, amssymb, graphicx, color}
\usepackage[bookmarksnumbered, colorlinks, plainpages]{hyperref}
\hypersetup{colorlinks=true,linkcolor=red, anchorcolor=green, citecolor=cyan, urlcolor=red, filecolor=magenta, pdftoolbar=true}

\usepackage[scr]{rsfso}

\newtheorem{theorem}{Theorem}[section]
\newtheorem{lemma}[theorem]{Lemma}
\newtheorem{proposition}[theorem]{Proposition}
\newtheorem{corollary}[theorem]{Corollary}
\theoremstyle{definition}
\newtheorem{definition}[theorem]{Definition}

\theoremstyle{remark}

\numberwithin{equation}{section}

\begin{document}

\setcounter{page}{1}

\title[Fractional Variable Exponents Sobolev Trace Spaces]{Fractional Variable Exponents Sobolev  Trace Spaces}

\author[MOHAMED BERGHOUT 
]{MOHAMED BERGHOUT $ ^{*} $\\
	Ibn Tofail University –Kenitra – Morocco\\
	B.P.242-Kenitra 14000 ESEF.\\
	Mohamed.berghout@uit.ac.ma\\
	 moh.berghout@gmail.com}



\dedicatory{\textbf {This paper is dedicated to my mother with deep estimate and  love} }

\subjclass[2010]{Primary 46E35,31C45, 31C15 }

\keywords{Fractional Sobolev spaces with variable exponents, trace spaces, relative capacity, quasicontinuity, removable sets}

\date{21 august 2020
\newline \indent $^{*}$Corresponding author}

\begin{abstract}
We introduce and study fractional variable exponents Sobolev trace spaces on any open set in the Euclidean space equipped with the Lebesgue measure. We show that every equivalence class of Sobolev functions has a quasicontinuous representatives. We use the relative capacity to characterize completely the zero trace fractional variable exponents Sobolev spaces. We also give a relative capacity criterium for removable sets.
\end{abstract} 
\maketitle
\section{\textbf {Introduction}}
The purpose of this paper is to introduce the fractional variable exponent Sobolev trace spaces and to characterize the traces of the Sobolev functions on any open set in the Euclidean space equipped with the Lebesgue measure. The motivation for this study is twofold. First we are interested in developing the theory of variable exponents Sobolev trace spaces in the fractional case, to this end, it is crucial that we can know the traces of Sobolev functions on the boundary of the set of definition. On the other hand, we would like to present a general theory which covers applications to fractional and nonlocal operators of elliptic type, see for example \cite{AB2,Del pezzo,UK} and references therein.

Functions spaces with variable exponent have been intensely investigated in the recent years. One of such spaces is the Lebesgue and Sobolev spaces with variable exponent.
They were introduced by W. Orlicz in 1931 \cite{Orlicz}; their properties were further developed by H. Nakano as special cases of the theory of modular spaces \cite{Nakano}. In the ensuing decades they were primarily considered as important examples of modular spaces or the class of Musielak–Orlicz spaces. In the beginning these spaces had theoretical interest. Later, at the end of the last century, their first use beyond the function spaces theory itself, was in variational problems and studies of $p(.)$-Laplacian operator, which in its turn gave an essential impulse for the development of this theory. For more details on these spaces,  see the monographs \cite{Hasto,kokivalish1,kokivalish2}.

We now give the main results of the paper. First, we introduce some notations which will be observed in this paper. Throughout this paper we will use the following notations: $\mathbb{R}^{n}$ is the $n$-dimens\-ional Euclidean space, and $n\in \mathbb{N}$ always stands for the dimension of the space. $\Omega \subset \mathbb{R}^{n}$ is a open set equipped with the $n$-dimensional Lebesgue measure. For constants we use the letter $C$ whose value may change even within a string of estimates. The ball with radius $ r $ and center $ x\in \mathbb{R}^{n}$  will be denoted by 
$ B(x,r) $. The closure of a set $A$ is denoted by $\overline{A}$ and the  topological boundary of $ A $ is denoted by $ \partial A $. The complement of $ A $ will be denoted by $ A^{c}$. We use the usual convention of identifying two $\mu$-measurable function on $A$ (a.e. in $A$, for short) if they agree almost everywhere, i.e. if they agree up to a set of $\mu$-measure zero. The characteristic function of a set $ E \subset A $ will be denoted by $ \chi_{E} $.  The Lebesgue integral of a Lebesgue measurable function $f:\Omega\longrightarrow \mathbb{R}$, is defined in the standard way and denoted by $\displaystyle\int_{\Omega} f(x)\ dx$. We use the symbol $ := $ to define the left-hand side by the right-hand side. For measurable functions $u,v: \Omega\longrightarrow \mathbb{R}$, we set  $u^{+}:=\max \left\lbrace u,0\right\rbrace $ and 
$u^{-}:=\max \left\lbrace -u,0\right\rbrace$. We denote by $L^{0}(\Omega)$ the space of all $\mathbb{R}$-valued measurable functions on $\Omega$. We denote by $\mathscr{C}(\Omega)$ the space of continuous functions on $\Omega$. By $\mathscr{C}_{c}(\overline{\Omega})$ we design the space of continuous functions on $\overline{\Omega}$ with compact support in $\overline{\Omega}$. We denote by $\mathscr{C}(\overline{\Omega})$ the space of uniformly continuous functions equipped with the supremum norm $\|f\|_{\infty}=\sup_{x\in\overline{\Omega}}|f(x)|$. By $\mathscr{C}^{k}(\overline{\Omega})$, $k\in \mathbb{N}$, we denote the space of all function $f$, such that $\partial_{\alpha} f :=\dfrac{\partial^{\left| \alpha \right|}f}{\partial^{\alpha_{1}}x_{1}........\partial^{\alpha_{n}}x_{n}} \in \mathscr{C}(\overline{\Omega})$ for all multi-index $\alpha = \left( \alpha_{1},\alpha_{2},...,\alpha_{n} \right)$,  $\left| \alpha \right|:=\alpha_{1}+\alpha_{2}+......+\alpha_{n} \leq k$. The space is equipped with the norm $\sup _{\left| \alpha \right|\leq k} {\left\| \partial_{\alpha} f\right\|}_{\infty}$, $\mathscr{C}^{\infty}(\overline{\Omega})=\bigcap_{k}\mathscr{C}^{k}(\overline{\Omega})$. The set of smooth functions in $\Omega$ is denoted by $\mathscr{C}^{\infty}(\Omega)$ - it consists of functions in $\Omega$ which are continuously differentiable arbitrarily many times. The set $\mathscr{C}^{\infty}_{0}(\Omega)$ is the subset of $\mathscr{C}^{\infty}(\Omega)$ of functions which have compact support.

Next, we introduce variable exponent Lebesgue and fractional Sobolev spaces as an abstract modular spaces. Let $ \Omega \subset\mathbb{R}^{n}$ be an open set. We fix $s\in\left( 0,1\right)$ and we consider two variable exponents, that is, $ q :{ \Omega}\rightarrow \left[ 1,+\infty \right) $ and $ p :{ \Omega}\times\Omega\rightarrow \left[ 1,+\infty \right) $ be two measurable functions. The set of variable exponents $ q :{\Omega}\rightarrow \left[ 1,+\infty \right)$ is denoted by $\mathcal{P}(\Omega)$ and the set of variable exponents  
$p:{ \Omega}\times\Omega\rightarrow \left[ 1,+\infty \right)$ is denoted by $\mathcal{P}(\Omega\times\Omega)$. we set $p^{-}:=$essinf$_{(x,y)\in\Omega\times\Omega}p(x,y)$,  $p^{+}:=$esssup$_{(x,y)\in\Omega\times\Omega}p(x,y)$, $q^{-}:=$essinf$_{x\in\Omega}q(x)$ and $q^{+}:=$esssup$_{x\in\Omega}q(x)$.\\
Throughout this paper we assume that 
\begin{eqnarray*}
	1<p^{-}\leq p(x,y)\leq p^{+} <\infty, 
\end{eqnarray*}
\begin{eqnarray*}
	1<q^{-}\leq q(x)\leq q^{+}<\infty. 
\end{eqnarray*} 
Notice that by \cite[ Proposition 4.1.7]{Hasto}, we can extend $q$ and $p$ to all of $\mathbb{R}^{n}$ and $\mathbb{R}^{n}\times \mathbb{R}^{n}$ respectively.

 The variable exponent Lebesgue space $ L^{p(.)}(\Omega) $ is the family of the equivalence classes of functions defined by
\begin{equation*}
	L^{p(.)}(\Omega):=\left\lbrace u\in L^{0}(\Omega):\rho_{p(.)}(\lambda u)=\int_{\Omega}\left|\lambda u(x)\right|^{p(x)} \ dx<\infty, \  
	\text{for some} \ \lambda>0\right\rbrace. 
\end{equation*}
The function $
\rho _{p(.)}: L^{p(.) }(\Omega)\longrightarrow
\left[ 0,\infty \right)$ is called the modular of the space 
$L^{p(.)}(\Omega)$. We define a norm, the so-called Luxembourg norm, in this space by 
\begin{equation*}
	\left\| u\right\|_{L^{p(.) }}=\inf \left\lbrace  \lambda >0:\rho
	_{p(.) }\left( \frac{u}{\lambda }\right) \leq 1\right\rbrace.
\end{equation*}
We define the fractional Sobolev space with variable exponents as follows:
\begin{eqnarray*}
	\begin{array}{l}
		\mathcal{W}^{s,q(.),p(.,.)}(\Omega):=
		\left\lbrace u  \in L^{q(.)}(\Omega): \dfrac{\left|u(x)-u(y)\right|  }{\left| x-y\right| ^{s+\frac{n}{p(.,.)}}}\in L^{p(.,.)}(\Omega\times\Omega) \right\rbrace.
	\end{array}
\end{eqnarray*}
We define a modular on $\mathcal{W}^{s,q(.),p(.,.)}(\Omega)$ by 
\begin{eqnarray*}
	\rho^{s,\Omega}_{q(.),p(.,.)}(u):=\int_{\Omega}\left| u(x)\right|^{q(x)}\ dx+\int_{\Omega}\int_{\Omega} \frac{\left|u(x)-u(y)\right|^{p(x,y)}}{\left|x-y\right|^{n+ sp(x,y)}}\ dx \  dy.	
\end{eqnarray*}
Let
\begin{eqnarray*}
	\left[ u\right]^{s,p(.,.)}(\Omega) :=\inf \left\lbrace\lambda>0: \displaystyle\int_{\Omega}\displaystyle\int_{\Omega}\dfrac{\left|u(x)-u(y)\right|^{p(x,y)}  }{\lambda^{p(x,y)}\left| x-y\right| ^{n+sp(x,y)}}dxdy \leq 1\right\rbrace,
\end{eqnarray*}
be the corresponding variable exponent Gagliardo semi-norm. It is easy to see that 
$\mathcal{W}^{s,q(.),p(.,.)}(\Omega)$ is a Banach space with the norm 
\begin{eqnarray*}
	\left\|u\right\|_{\mathcal{W}^{s,q(.),p(.,.)}(\Omega)} :=\|u\rVert_{L^{q(.)}(\Omega)}+ \left[ u\right]^{s,p(.,.)}(\Omega).
\end{eqnarray*}
It is clear that ${\mathcal{W}^{s,q(.),p(.,.)}(\Omega)}$ can be seen as a natural extension of the classical fractional Sobolev space. The modular 
$ \rho^{s,\Omega}_{q(.),p(.,.)} $ induces a norm by
\begin{eqnarray*}
	\left\|u\right\|_{ \rho^{s,\Omega}_{q(.),p(.,.)}}:=\inf \left\lbrace \lambda>0: \rho^{s,\Omega}_{q(.),p(.,.)}(\frac{1}{\lambda}u)\leqslant1\right\rbrace,   	
\end{eqnarray*} 
which is equivalent to the norm $\left\|u\right\|_{\mathcal{W}^{s,q(.),p(.,.)}(\Omega)}$. It is clear that $u\in \mathcal{W}^{s,q(.),p(.,.)}(\Omega)$ if and only if 
$ \rho^{s,\Omega}_{q(.),p(.,.)}(u)<\infty$.
 
Note that the space $\mathcal{W}^{s,q(.),p(.,.)}(\Omega)$ is a separable, reflexive and uniformly convex Banach lattice space, we refer to \cite{BA3}.

The definition of the spaces 
$L^{q(.)}(\mathbb{R}^{n})$ and $\mathcal{W}^{s,q(.),p(.,.)}(\mathbb{R}^{n})$ is analogous to $L^{q(.)}(\Omega)$ and $W^{s,q(.),p(.,.)}(\Omega)$; one just changes every occurrence of $\Omega$ by $\mathbb{R}^{n}$. We refer to \cite{BA1,BA,BA2,BA3,UK}.

In the sequel, we define the fractional variable exponent Sobolev zero trace spaces in a proper open subset $ \Omega $ of $ \mathbb{R}^{n}$.
\begin{definition}
	Let $s\in(0,1), q\in \mathcal{P}(\Omega)$ and 
	$p\in\mathcal{P}(\Omega\times\Omega)$. The fractional Sobolev space $\mathcal{W}^{s,q(.),p(.,.)}_{0}(\Omega)$ is the closure of the set of $\mathcal{W}^{s,q(.),p(.,.)}(\Omega)$-functions with compact support, i.e.
	\begin{equation*}
	\left\lbrace u\in \mathcal{W}^{s,q(.),p(.,.)}(\Omega):u=u\chi_{K} \ \text{for a compact} \ K\subset \Omega  \right\rbrace. 
	\end{equation*}
	in $\mathcal{W}^{s,q(.),p(.,.)}(\Omega)$. 
\end{definition}
A subspace $I$ of $ \mathcal{W}^{s,q(.),p(.,.)}(\Omega)$ is called an ideal if for 
$u\in I,\ v\in \mathcal{W}^{s,q(.),p(.,.)}(\Omega)$, $ \left|v \right| \leqslant \left|u \right|$ a.e. implies that $v\in I$. A (real valued) function space is a lattice if the point-wise minimum and maximum of any two of its elements belong to the space. The closed lattice ideals of the Sobolev spaces $\mathcal{W}^{s,q(.),p(.,.)}(\Omega)$ are those subspaces which consist of all functions which vanish on a prescribed set. To be precise, we have the following result:
\begin{theorem}
	Let $s\in(0,1), q\in \mathcal{P}(\Omega)$ and 
	$p\in\mathcal{P}(\Omega\times\Omega)$. Then the space $\mathcal{W}^{s,q(.),p(.,.)}_{0}(\Omega)$ is a closed ideal in $\mathcal{W}^{s,q(.),p(.,.)}(\Omega)$.
\end{theorem}
We denote by $ \mathcal{{H}}^{s,q(.),p(.,.)}_{0}(\Omega) $ the closure of $\mathscr{C}_{0}^{\infty}(\Omega)$ in 
$\mathcal{W}^{s,q(.),p(.,.)}(\Omega)$. For $s\in(0,1), q\in \mathcal{P}(\Omega)$ and $p\in\mathcal{P}(\Omega\times\Omega)$, we let
\begin{eqnarray*}
	\mathcal{\tilde{{W}}}^{s,q(.),p(.,.)}(\Omega):= \overline{{\mathcal{W}}^{s,q(.),p(.,.)}(\Omega)\cap \mathscr{C}_{c}(\overline{\Omega})}^{{\mathcal {W}}^{s,q(.),p(.,.)}(\Omega)}
\end{eqnarray*} 
The following result shows that the spaces $\mathcal{{H}}^{s,q(.),p(.,.)}_{0}(\Omega)$, $ \mathcal{\tilde{{W}}}^{s,q(.),p(.,.)}(\Omega) $ and $ \mathcal{W}^{s,q(.),p(.,.)}_{0}(\Omega) $ may coincide.
 \begin{theorem}
	Let $s\in(0,1), q\in \mathcal{P}(\Omega)$ and 
	$p\in\mathcal{P}(\Omega\times\Omega)$. Assume that $ \mathcal{W}^{s,q(.),p(.,.)}(\Omega)\cap \mathscr{C}^{\infty}(\overline{\Omega})$ is dense in 
	$ \mathcal{W}^{s,q(.),p(.,.)}(\Omega)$. 
	Then
	\begin{equation*}
	\mathcal{{H}}^{s,q(.),p(.,.)}_{0}(\Omega)= \mathcal{W}^{s,q(.),p(.,.)}_{0}(\Omega)=\mathcal{\tilde{{W}}}^{s,q(.),p(.,.)}(\Omega).
	\end{equation*}  
\end{theorem}
Now, we introduce the fractional Sobolev $(s,q(.),p(.,.))$-capacity, in which the capacity of a set is taken relative to a open subset. 
Let $O\subset \overline{\Omega}$ be a relatively open set, that is, open with respect to the relative topology of $\overline{\Omega}$. Let $s\in(0,1), q\in \mathcal{P}(\Omega)$ and $p\in\mathcal{P}(\Omega\times\Omega)$. We denote
\begin{eqnarray*}
	\mathcal{R}^{s, \overline{\Omega}}_{q(.),p(.,.)}(O):=\left\lbrace u\in \tilde{{W}}^{s,q(.),p(.,.)}(\Omega): u\geq 1 \ \textnormal{ a.e. on}\ O \right\rbrace.    	
\end{eqnarray*} 
We define the fractional relative $(s,q(.),p(.,.))$-capacity of $O$, with respect to $\Omega$, by
\begin{eqnarray*}
	C_{q(.),p(.,.)}^{s,\Omega}(O):=\inf_{u\in\mathcal{R}^{s, \overline{\Omega}}_{q(.),p(.,.)}(O)}{\rho}_{q(.),p(.,.)}^{s,\Omega}(u).	   	
\end{eqnarray*}
For any set $ E\subset \overline{\Omega}$,
\begin{eqnarray*} 
	C_{q(.),p(.,.)}^{s,\Omega}(E):=\inf\left\lbrace C_{q(.),p(.,.)}^{s,\Omega}(O): O \ \textnormal{relatively open in}\ \overline{\Omega} \ \textnormal{containing}\ E  \right\rbrace. 
\end{eqnarray*}
Next, we give a characterization of $\mathcal{{H}}^{s,q(.),p(.,.)}_{0}(\Omega)$
and a necessary and sufficient condition in terms of the $ (s,q(.),p(.,.))$-relative capacity for the equality
$\mathcal{{H}}^{s,q(.),p(.,.)}_{0}(\Omega)=\mathcal{\tilde{{W}}}^{s,q(.),p(.,.)}(\Omega)$ to be true.
\begin{theorem}
	Let $\Omega\subset \mathbb{R}^{n}$ be an open set, $s\in(0,1)$, $q\in \mathcal{P}(\Omega)$ and 
	$p\in\mathcal{P}(\Omega\times\Omega)$. Then
	\begin{eqnarray*}
		\mathcal{H}^{s,q(.),p(.,.)}_{0}(\Omega)=\left\lbrace u\in \tilde{\mathcal{W}}^{s,q(.),p(.,.)}(\Omega):\ \tilde{u}=0 \ (s,q(.),p(.,.))\textnormal{-r.q.e. on} \ \partial\Omega   \right\rbrace, 	
	\end{eqnarray*} 
	where $\tilde{u}$ denotes the $(s,q(.),p(.,.))$-relatively quasicontinuous representative of $u$.
\end{theorem} 
\begin{corollary}
	Let $\Omega\subset \mathbb{R}^{n}$ be open. Then the following assertions are equivalent.
	\begin{enumerate}
		\item $ C_{q(.),p(.,.)}^{s,\Omega}(\partial\Omega)=0$;
		\item $\mathcal{H}^{s,q(.),p(.,.)}_{0}(\Omega)=\tilde{\mathcal{W}}^{s,q(.),p(.,.)}(\Omega)$.
	\end{enumerate}
\end{corollary}
A function $ u:\Omega\longrightarrow \mathbb{R}$ is said to be $(s,q(.),p(.,.))$-relatively quasicontinuous ($ (s,q(.),p(.,.)) $-r.q.c., for short) if for every $ \varepsilon>0 $, there exists a relatively open set $ O_{\varepsilon}\subset\Omega $ such that $C_{q(.),p(.,.)}^{s,\Omega}(O_{\varepsilon})<\varepsilon$ and $ u $ is continuous on $ \Omega\setminus O_{\varepsilon}$. Now, we gives another definition for fractional variable exponents Sobolev zero trace spaces $ \tilde{\mathcal{W}}^{s,q(.),p(.,.)}_{0}(\Omega) $ in view of potential theory.   
\begin{definition}
	Let $s\in(0,1), q\in \mathcal{P}(\Omega)$ and 
	$p\in\mathcal{P}(\Omega\times\Omega)$. We denote
	$ u \in \tilde{\mathcal{W}}^{s,q(.),p(.,.)}_{0}(\Omega)$ and say that $ u $ belongs to the fractional variable exponents Sobolev zero trace spaces $ \tilde{\mathcal{W}}^{s,q(.),p(.,.)}_{0}(\Omega) $ if there exists a $(s,q(.),p(.,.))$-relatively quasicontinuous function $ \tilde{u}\in\tilde{\mathcal{W}}^{s,q(.),p(.,.)}(\Omega)$ such
	that $ \tilde{u}= u$ a.e. in $\Omega$  and $ \tilde{u}=0$ $(s,q(.),p(.,.))$-r.q.e. in $ \Omega^{c} $.
\end{definition}
It is important to see the connections between all of the previous spaces. In fact under the density assumption they coincides. More precisely, we have the following theorem:
 \begin{theorem}
	Let $s\in(0,1), q\in \mathcal{P}(\Omega)$ and 
	$p\in\mathcal{P}(\Omega\times\Omega)$. Assume that $ \mathcal{W}^{s,q(.),p(.,.)}(\Omega)\cap \mathscr{C}^{\infty}(\overline{\Omega})$ is dense in 
	$ \mathcal{W}^{s,q(.),p(.,.)}(\Omega)$. 
	Then
	\begin{equation*}
	\mathcal{{H}}^{s,q(.),p(.,.)}_{0}(\Omega)=\mathcal{\tilde{{W}}}^{s,q(.),p(.,.)}_{0}(\Omega)= \mathcal{W}^{s,q(.),p(.,.)}_{0}(\Omega)=\mathcal{\tilde{{W}}}^{s,q(.),p(.,.)}(\Omega).
	\end{equation*}  
\end{theorem}
It is natural to ask for which a subsets $ N $ of $ \Omega $ are removable for the Sobolev space $\mathcal{\tilde{{W}}}^{s,q(.),p(.,.)}_{0}(\Omega)$.
\begin{definition}
Let $ \Omega\subset \mathbb{R}^{n}$ be an open set. We call a subset $ N $ of $ \Omega $ is removable for $\mathcal{\tilde{{W}}}^{s,q(.),p(.,.)}_{0}(\Omega)$ if 
\begin{equation*}
\mathcal{\tilde{{W}}}^{s,q(.),p(.,.)}_{0}(\Omega)=\mathcal{\tilde{{W}}}^{s,q(.),p(.,.)}_{0}(\Omega\setminus N).	
\end{equation*}  
\end{definition}
The next theorem gives a relative capacity criterium for removable subsets for $\tilde{\mathcal{W}}^{s,q(.),p(.,.)}_{0}(\Omega)$.
\begin{theorem}
	Let $ N $ be a subset of $\Omega$. Then the following assertions are equivalent.
	\begin{enumerate}
		\item $ C_{q(.),p(.,.)}^{s,\Omega}(N\cap\Omega)=0$;
		\item $\tilde{\mathcal{W}}^{s,q(.),p(.,.)}_{0}(\Omega)=\tilde{\mathcal{W}}^{s,q(.),p(.,.)}_{0}(\Omega\setminus N)$.
	\end{enumerate}
\end{theorem}

We now give the structure of the paper. More detailed descriptions appear at the beginnings of the sections.

Section $ 2 $ gives sufficient conditions for the existence of a $(s,q(.),p(.,.))$-relatively quasicontinuous representative for functions in $\tilde{\mathcal{W}}^{s,q(.),p(.,.)}(\Omega)$.

Section $ 3 $ modulo the technical results of Section $ 2 $ develop a theory of fractional variable exponents Sobolev trace spaces and gives a necessary and sufficient condition in terms of $(s,q(.),p(.,.))$-relative capacity of removable subsets  for $\tilde{\mathcal{W}}^{s,q(.),p(.,.)}_{0}(\Omega)$.
\section{\textbf{ $(s,q(.),p(.,.))$-quasicontinuous representative  of equivalence class of Sobolev functions}}
Sobolev functions are defined only up to Lebesgue measure zero and thus it is not always clear how to use their point-wise properties. But one can also think of some representative of this equivalence class, perhaps defined at all points outside a set of measure zero.

In this section, we show that the equivalent class of Sobolev functions in $\tilde{\mathcal{W}}^{s,q(.),p(.,.)}(\Omega)$ are  
$ (s,q(.),p(.,.))$-relatively quasicontinuous, it turns out that Sobolev functions are defined up to a set of $ (s,q(.),p(.,.))$-relative capacity zero. This is a concept with deep roots when we want to study fractional variable exponents spaces with zero boundary values.

We begin by recalling some proprieties of the fractional Sobolev $(s,q(.),p(.,.))$-capacity appearing in the existing literature, we refer to \cite{BA3,Warma1}.\\
First, we recall the definition of the fractional relative $(s,q(.),p(.,.))$-capacity.
	Let $O\subset \overline{\Omega}$ be a relatively open set, that is, open with respect to the relative topology of $\overline{\Omega}$. Let $s\in(0,1), q\in \mathcal{P}(\Omega)$ and $p\in\mathcal{P}(\Omega\times\Omega)$. We denote
	\begin{eqnarray*}
		\mathcal{R}^{s, \overline{\Omega}}_{q(.),p(.,.)}(O):=\left\lbrace u\in \tilde{{\mathcal{W}}}^{s,q(.),p(.,.)}(\Omega): u\geq 1 \ \textnormal{ a.e. on}\ O \right\rbrace.    	
	\end{eqnarray*} 
	We define the fractional relative $(s,q(.),p(.,.))$-capacity of $O$, with respect to $\Omega$, by
	\begin{eqnarray*}
		C_{q(.),p(.,.)}^{s,\Omega}(O):=\inf_{u\in\mathcal{R}^{s, \overline{\Omega}}_{q(.),p(.,.)}(O)}{\rho}_{q(.),p(.,.)}^{s,\Omega}(u).	   	
	\end{eqnarray*}
	For any set $ E\subset \overline{\Omega}$,
	\begin{eqnarray*} 
		C_{q(.),p(.,.)}^{s,\Omega}(E):=\inf\left\lbrace C_{q(.),p(.,.)}^{s,\Omega}(O): O \ \textnormal{relatively open in}\ \overline{\Omega} \ \textnormal{containing}\ E  \right\rbrace. 
	\end{eqnarray*}
Recall from \cite{BA3} that the set function $E\mapsto 	C_{q(.),p(.,.)}^{s,\Omega}$ has the following properties:
 \begin{itemize}
 	\item $(C_{1})$ $	C_{q(.),p(.,.)}^{s,\Omega}(\varnothing)=0$;
 	\item $(C_{2})$ If $ E_{1}\subset E_{2} \subset \Omega_{2} \subset \Omega_{1}$, then
 	\begin{equation*}
 		C_{q(.),p(.,.)}^{s,\Omega_{1}}(E_{1})\leqslant 	C_{q(.),p(.,.)}^{s,\Omega_{2}}(E_{2});
 	\end{equation*}
 	\item $(C_{3})$ If $K_{1}\supset K_{2}\supset K_{3}\dots$ are compact subsets of $\Omega$, then
 	\begin{equation*}
 C_{q(.),p(.,.)}^{s,\Omega}\left( \bigcap_{i=1}^{\infty} K_{i}\right) =\lim _{i\rightarrow\infty}C_{q(.),p(.,.)}^{s,\Omega}(K_{i}); 
 	\end{equation*}
 	\item $(C_{4})$ If $E_{1}\subset E_{2}\dots$ are subsets of $\Omega$, then
 	\begin{equation*}
 	C_{q(.),p(.,.)}^{s,\Omega}\left(\bigcup_{i=1}^{\infty} E_{i}\right) =\lim _{i\rightarrow\infty}C_{q(.),p(.,.)}^{s,\Omega}(E_{i}); 
 	\end{equation*}
 	\item $(C_{5})$ For $E_{i}\subset \Omega$, $ i\in \mathbb{N}$, we have
 	\begin{equation*}
 	C_{q(.),p(.,.)}^{s,\Omega}\left( \bigcup_{i=1}^{\infty} E_{i}\right) \leqslant \sum ^{\infty}_{i=1} C_{q(.),p(.,.)}^{s,\Omega}(E_{i}).
 	\end{equation*}
 \end{itemize}
 This means that the fractional relative $(s,q(.),p(.,.))$-capacity $ C_{q(.),p(.,.)}^{s,\Omega}$ is an outer measure and a Choquet capacity. 
 
 Next, we shows that the fractional relative $(s,q(.),p(.,.))$-capacity is strongly subadditive. Let us beginning by the following useful lemma.
 \begin{lemma} \label{inegualite modulaire}
 Let $s\in(0,1), q\in \mathcal{P}(\Omega)$ and $p\in\mathcal{P}(\Omega\times\Omega)$. Let $u_{1}, u_{2}\in\mathcal{W}^{s,q(.),p(.,.)}(\Omega)$ be nonnegative. We set $ u:=\max\left\lbrace u_{1},u_{2} \right\rbrace$ and $ v:=\min\left\lbrace u_{1},u_{2} \right\rbrace$. Then $u, v\in\mathcal{W}^{s,q(.),p(.,.)}(\Omega)$ and 
 \begin{equation*}
 \rho_{q(.),p(.,.)}^{s,\Omega}(u)+ \rho_{q(.),p(.,.)}^{s,\Omega}(v)\leqslant  \rho_{q(.),p(.,.)}^{s,\Omega}(u_{1})+\rho_{q(.),p(.,.)}^{s,\Omega}(u_{2}).	
 \end{equation*} 
 \end{lemma}
 \begin{proof}
 Let $ s,q,p,u_{1},u_{2},u$ and $ v $ be as in the statement of the lemma. By \cite[Proposition 2 ]{BA3}, we have that $u, v\in\mathcal{W}^{s,q(.),p(.,.)}(\Omega)$. Let 
 \begin{equation*}
  \Omega_{1}:=\left\lbrace x\in\Omega; u_{1}(x)\leqslant u_{2}(x) \right\rbrace \ \text{and} \   \Omega_{2}:=\left\lbrace x\in\Omega; u_{1}(x)> u_{2}(x) \right\rbrace.
 \end{equation*}
 Then
 \begin{equation*}
\int_{\Omega}\left| u(x)\right|^{q(x)}\ dx=\int_{\Omega_{1}}\left| u_{2}(x)\right|^{q(x)}\ dx+\int_{\Omega_{2}}\left| u_{1}(x)\right|^{q(x)}\ dx  	
 \end{equation*}
and
 \begin{equation*}
	\int_{\Omega}\left| v(x)\right|^{q(x)}\ dx=\int_{\Omega_{1}}\left| u_{1}(x)\right|^{q(x)}\ dx+\int_{\Omega_{2}}\left| u_{2}(x)\right|^{q(x)}\ dx  	
\end{equation*}
Hence
 \begin{equation*}
	\int_{\Omega}\left| u(x)\right|^{q(x)}\ dx+\int_{\Omega}\left| v(x)\right|^{q(x)}\ dx=\int_{\Omega}\left| u_{1}(x)\right|^{q(x)}\ dx+\int_{\Omega}\left| u_{2}(x)\right|^{q(x)}\ dx.  	
\end{equation*}
On the other hand,
\begin{eqnarray*}
\int_{\Omega}\int_{\Omega} \frac{\left|u(x)-u(y)\right|^{p(x,y)}}{\left|x-y\right|^{n+ sp(x,y)}}\ dx \  dy=\int_{\Omega_{1}}\int_{\Omega_{1}}\frac{\left|u_{2}(x)-u_{2}(y)\right|^{p(x,y)}}{\left|x-y\right|^{n+ sp(x,y)}}\ dx \  dy\\
+
\int_{\Omega_{1}}\int_{\Omega_{2}}\frac{\left|u_{1}(x)-u_{2}(y)\right|^{p(x,y)}}{\left|x-y\right|^{n+ sp(x,y)}}\ dx \  dy + \int_{\Omega_{2}}\int_{\Omega_{1}}\frac{\left|u_{2}(x)-u_{1}(y)\right|^{p(x,y)}}{\left|x-y\right|^{n+ sp(x,y)}}\ dx \  dy \\ 
+ \int_{\Omega_{2}}\int_{\Omega_{2}}\frac{\left|u_{1}(x)-u_{1}(y)\right|^{p(x,y)}}{\left|x-y\right|^{n+ sp(x,y)}}\ dx \  dy  
\end{eqnarray*}
and
\begin{eqnarray*}
	\int_{\Omega}\int_{\Omega} \frac{\left|v(x)-v(y)\right|^{p(x,y)}}{\left|x-y\right|^{n+ sp(x,y)}}\ dx \  dy=\int_{\Omega_{1}}\int_{\Omega_{1}}\frac{\left|u_{1}(x)-u_{1}(y)\right|^{p(x,y)}}{\left|x-y\right|^{n+ sp(x,y)}}\ dx \  dy\\
	+
	\int_{\Omega_{1}}\int_{\Omega_{2}}\frac{\left|u_{2}(x)-u_{1}(y)\right|^{p(x,y)}}{\left|x-y\right|^{n+ sp(x,y)}}\ dx \  dy + \int_{\Omega_{2}}\int_{\Omega_{1}}\frac{\left|u_{1}(x)-u_{2}(y)\right|^{p(x,y)}}{\left|x-y\right|^{n+ sp(x,y)}}\ dx \  dy \\ 
	+ \int_{\Omega_{2}}\int_{\Omega_{2}}\frac{\left|u_{2}(x)-u_{2}(y)\right|^{p(x,y)}}{\left|x-y\right|^{n+ sp(x,y)}}\ dx \  dy.  
\end{eqnarray*}
Hence 
\begin{eqnarray*}
\int_{\Omega}\int_{\Omega} \frac{\left|u(x)-u(y)\right|^{p(x,y)}}{\left|x-y\right|^{n+ sp(x,y)}}\ dx \  dy	+\int_{\Omega}\int_{\Omega} \frac{\left|v(x)-v(y)\right|^{p(x,y)}}{\left|x-y\right|^{n+ sp(x,y)}}\ dx \  dy=\\
 \int_{\Omega_{2}}\int_{\Omega_{2}}\frac{\left|u_{1}(x)-u_{1}(y)\right|^{p(x,y)}}{\left|x-y\right|^{n+ sp(x,y)}}\ dx \  dy +\int_{\Omega_{1}}\int_{\Omega_{1}}\frac{\left|u_{1}(x)-u_{1}(y)\right|^{p(x,y)}}{\left|x-y\right|^{n+ sp(x,y)}}\ dx \  dy\\
 +\int_{\Omega_{2}}\int_{\Omega_{1}}\frac{\left|u_{2}(x)-u_{1}(y)\right|^{p(x,y)}+\left|u_{1}(x)-u_{2}(y)\right|^{p(x,y)}}{\left|x-y\right|^{n+ sp(x,y)}}\ dx \  dy\\
 +\int_{\Omega_{1}}\int_{\Omega_{1}}\frac{\left|u_{2}(x)-u_{2}(y)\right|^{p(x,y)}}{\left|x-y\right|^{n+ sp(x,y)}}\ dx \  dy+\int_{\Omega_{2}}\int_{\Omega_{2}}\frac{\left|u_{2}(x)-u_{2}(y)\right|^{p(x,y)}}{\left|x-y\right|^{n+ sp(x,y)}}\ dx \  dy\\
 +\int_{\Omega_{1}}\int_{\Omega_{2}}\frac{\left|u_{1}(x)-u_{2}(y)\right|^{p(x,y)}+\left|u_{2}(x)-u_{1}(y)\right|^{p(x,y)}}{\left|x-y\right|^{n+ sp(x,y)}}\ dx \  dy  
\end{eqnarray*}
Now, by using \cite[ Lemma 3.3]{Warma} with the mapping $ F:\mathbb{R}^{2}\longrightarrow \left[0,+\infty \right)$ defined by $F_{p(x,y)}(\zeta,\beta)=\left|\zeta-\beta\right|^{p(x,y)}$, we get that
\begin{equation*}
\left|u_{2}(x)-u_{1}(y)\right|^{p(x,y)}+\left|u_{1}(x)-u_{2}(y)\right|^{p(x,y)}\leqslant \left|u_{1}(x)-u_{1}(y)\right|^{p(x,y)}+\left|u_{2}(x)-u_{2}(y)\right|^{p(x,y)}
\end{equation*}
on $\Omega_{1}\times \Omega_{2}:=\left\lbrace (x,y)\in\Omega\times\Omega,x\in\Omega_{1},y\in\Omega_{2} \right\rbrace$ and 
\begin{equation*}
\left|u_{1}(x)-u_{2}(y)\right|^{p(x,y)}+\left|u_{2}(x)-u_{1}(y)\right|^{p(x,y)}\leqslant \left|u_{1}(x)-u_{1}(y)\right|^{p(x,y)}+\left|u_{2}(x)-u_{2}(y)\right|^{p(x,y)}
\end{equation*}
on $\Omega_{2}\times \Omega_{1}:=\left\lbrace (x,y)\in\Omega\times\Omega,x\in\Omega_{2},y\in\Omega_{1} \right\rbrace$. Hence
\begin{eqnarray*}
	\int_{\Omega}\int_{\Omega} \frac{\left|u(x)-u(y)\right|^{p(x,y)}}{\left|x-y\right|^{n+ sp(x,y)}}\ dx \  dy	+\int_{\Omega}\int_{\Omega} \frac{\left|v(x)-v(y)\right|^{p(x,y)}}{\left|x-y\right|^{n+ sp(x,y)}}\ dx \  dy\\
	 \leqslant 	\int_{\Omega}\int_{\Omega} \frac{\left|u_{1}(x)-u_{1}(y)\right|^{p(x,y)}}{\left|x-y\right|^{n+ sp(x,y)}}\ dx \  dy	+\int_{\Omega}\int_{\Omega} \frac{\left|u_{2}(x)-u_{2}(y)\right|^{p(x,y)}}{\left|x-y\right|^{n+ sp(x,y)}}\ dx \  dy
\end{eqnarray*}
and 
\begin{equation*}
\int_{\Omega}\left| u(x)\right|^{q(x)}\ dx+\int_{\Omega}\left| v(x)\right|^{q(x)}\ dx=\int_{\Omega}\left| u_{1}(x)\right|^{q(x)}\ dx+\int_{\Omega}\left| u_{2}(x)\right|^{q(x)}\ dx.  	
\end{equation*}
Consequently, we get that
\begin{equation*}
\rho_{q(.),p(.,.)}^{s,\Omega}(u)+ \rho_{q(.),p(.,.)}^{s,\Omega}(v)\leqslant  \rho_{q(.),p(.,.)}^{s,\Omega}(u_{1})+\rho_{q(.),p(.,.)}^{s,\Omega}(u_{2}).	
\end{equation*} 
\end{proof}  
We notice that Lemma \ref{inegualite modulaire} remain true if one replaces $\mathcal{{W}}^{s,q(.),p(.,.)}(\Omega)$ with the space $ \mathcal{\tilde{{W}}}^{s,q(.),p(.,.)}(\Omega)$. 
\begin{proposition}\label{strong subbaditivity de capacity}
 The fractional relative $(s,q(.),p(.,.))$-capacity is strongly subadditive, that is, for all $ A, B \subset\overline{\Omega}$,
 \begin{equation*}
 C_{q(.),p(.,.)}^{s,\Omega}(A\cup B)+C_{q(.),p(.,.)}^{s,\Omega}(A\cap B)\leqslant C_{q(.),p(.,.)}^{s,\Omega}(A)+C_{q(.),p(.,.)}^{s,\Omega}(B). 	
 \end{equation*}
\end{proposition}
\begin{proof}
Let $ A $ and $ B $ be two subsets of $ \overline{\Omega}$. Let $\varepsilon>0$, $ u_{1}\in\mathcal{R}^{s, \overline{\Omega}}_{q(.),p(.,.)}(A)$ and $ u_{2}\in\mathcal{R}^{s, \overline{\Omega}}_{q(.),p(.,.)}(B)$ such that 
\begin{equation*}
\rho_{q(.),p(.,.)}^{s,\Omega}(u_{1})\leqslant C_{q(.),p(.,.)}^{s,\Omega}(A)+\dfrac{\varepsilon}{2}
\end{equation*}
and
\begin{equation*}
\rho_{q(.),p(.,.)}^{s,\Omega}(u_{2})\leqslant C_{q(.),p(.,.)}^{s,\Omega}(B)+\dfrac{\varepsilon}{2}.
\end{equation*}
By Lemma \ref{inegualite modulaire}, we have that $\max\left\lbrace u_{1},u_{2}\right\rbrace$ and $\min\left\lbrace u_{1},u_{2}\right\rbrace$ are in $ \mathcal{\tilde{{W}}}^{s,q(.),p(.,.)}(\Omega)$ and 
\begin{equation*}
\rho_{q(.),p(.,.)}^{s,\Omega}(\max\left\lbrace u_{1},u_{2}\right\rbrace)+ \rho_{q(.),p(.,.)}^{s,\Omega}(\min\left\lbrace u_{1},u_{2}\right\rbrace)\leqslant  \rho_{q(.),p(.,.)}^{s,\Omega}(u_{1})+\rho_{q(.),p(.,.)}^{s,\Omega}(u_{2}).	
\end{equation*}
Since $ \max\left\lbrace u_{1},u_{2}\right\rbrace \in\mathcal{R}^{s, \overline{\Omega}}_{q(.),p(.,.)}(A\cup B)$ and $ \min\left\lbrace u_{1},u_{2}\right\rbrace \in\mathcal{R}^{s, \overline{\Omega}}_{q(.),p(.,.)}(A\cap B)$, it follows that,
\begin{eqnarray*}
C_{q(.),p(.,.)}^{s,\Omega}(A\cup B)+C_{q(.),p(.,.)}^{s,\Omega}(A\cap B)&\leqslant& \rho_{q(.),p(.,.)}^{s,\Omega}(u_{1})+\rho_{q(.),p(.,.)}^{s,\Omega}(u_{2})\\
&\leqslant& C_{q(.),p(.,.)}^{s,\Omega}(A)+C_{q(.),p(.,.)}^{s,\Omega}(B)+\varepsilon, 	
\end{eqnarray*}
which yields the claim, as $ \varepsilon $ tend to zero.     
\end{proof}     
\begin{definition}
	\item A set $P\subset\Omega$ is called $ (s,q(.),p(.,.))$-relatively polar if $C_{q(.),p(.,.)}^{s,\Omega}(P)=0$.
\item We say that a property holds on a  set $ A\subset\Omega$ 
$ (s,q(.),p(.,.))$-relatively quasi everywhere ($ (s,q(.),p(.,.)) $-r.q.e., for short) if there exists a $(s,q(.),p(.,.))$-relatively polar set $P\subset A$ such that the property holds everywhere on $ A\setminus P$.
\end{definition}
By definition $\mathcal{W}^{s,q(.),p(.,.)}(\Omega)\cap \mathscr{C}_{c}(\overline{\Omega})$ is dense in $ \tilde{{\mathcal{W}}}^{s,q(.),p(.,.)}(\Omega)$ which is complete Banach space. The next result gives a way to find a $ (s,q(.),p(.,.)) $-relatively quasi everywhere converging subsequence.
\begin{theorem}\label{subsequence which converges $p(.)$-r.q.e.}
 Let $s\in(0,1), q\in \mathcal{P}(\Omega)$ and 
	$p\in\mathcal{P}(\Omega\times\Omega)$. For each Cauchy sequence with respect to the $\mathcal{W}^{s,q(.),p(.,.)}(\Omega)$-norm of functions in $\tilde{\mathcal{W}}^{s,q(.),p(.,.)}(\Omega)$ there exists a subsequence which converges $(s,q(.),p(.,.))$-r.q.e. in $\Omega$. Moreover, the convergence is uniform outside a set of arbitrary small relative $(s,q(.),p(.,.))$-capacity.
\end{theorem}
\begin{proof}
	Let $(u_{i})$ be a Cauchy sequence in $\tilde{\mathcal{W}}^{s,q(.),p(.,.)}(\Omega)$. Without loss of generality, we denote again by $(u_{i})$ the subsequence of $(u_{i})$ such that
	\begin{eqnarray*}
		\left\| u_{i+1}-u_{i} \right\|_{\mathcal{W}^{s,q(.),p(.,.)}(\Omega)}\leq \dfrac{1}{8^{i}}, \ i\in \mathbb{N}.   
	\end{eqnarray*}
	Put
	\begin{eqnarray*}
		G_{i}:=\left\lbrace x\in\Omega: \ \left|u_{i+1}(x)-u_{i}(x) \right|>\dfrac{1}{2^{i}}, \ i\in \mathbb{N}\right\rbrace   
	\end{eqnarray*}
	and 
	\begin{eqnarray*}
		G_{k}:=\bigcup_{i=k}^{\infty}G_{i}.  
	\end{eqnarray*}
	Hence $G_{i}$ is an open set in $\Omega$ and $ 2^{i} \left|u_{i+1}(x)-u_{i}(x) \right|>1$ on $G_{i}$. Since 
	$\mathcal{W}^{s,q(.),p(.,.)}(\Omega)$ is a Banach lattice, we deduce that $\tilde{\mathcal{W}}^{s,q(.),p(.,.)}(\Omega)$ is also a Banach lattice. Therefore, 
	\begin{eqnarray*}
		2^{i} \left|u_{i+1}(x)-u_{i}(x) \right|\in\tilde{\mathcal{W}}^{s,q(.),p(.,.)}(\Omega)  
		\ \textnormal{and}\ \left\| 2^{i} \left|u_{i+1}(x)-u_{i}(x) \right|\right\|_{\mathcal{W}^{s,q(.),p(.,.)}(\Omega)}\leqslant \dfrac{1}{4^{i}}\leqslant 1. 
	\end{eqnarray*}
	By the unit ball property (\cite[Lemma 2.1.14]{Hasto}), we get that
	\begin{eqnarray*}
		\rho_{q(.),p(.,.)}^{s,\Omega}(2^{i} \left|u_{i+1}(x)-u_{i}(x) \right|)\leqslant 
		2^{i}\left\| u_{i+1}-u_{i}\right\|_{\mathcal{W}^{s,q(.),p(.,.)}(\Omega)}\leqslant \dfrac{1}{4^{i}}. 
	\end{eqnarray*}
	Consequently
	\begin{eqnarray*}
	C_{q(.),p(.,.)}^{s,\Omega}(G_{i})\leqslant \dfrac{1}{4^{i}}.
	\end{eqnarray*}
	By property $(C_{5})$ of the fractional relative $(s,q(.),p(.,.))$-capacity, we obtain that
	\begin{eqnarray*}
		C_{q(.),p(.,.)}^{s,\Omega}(G_{k})&=&C_{q(.),p(.,.)}^{s,\Omega}(\cup^{\infty}_{i=k} G_{i})\\
		&\leqslant& \sum_{i=k}^{\infty}C_{q(.),p(.,.)}^{s,\Omega}(G_{i})\\
		&\leqslant& \sum_{i=k}^{\infty} \dfrac{1}{4^{i}}=\dfrac{1}{4^{k-1}}.
	\end{eqnarray*}
	Hence
	\begin{eqnarray*}
		C_{q(.),p(.,.)}^{s,\Omega}(\cap_{k=1}^{\infty} G_{k})&=&C_{q(.),p(.,.)}^{s,\Omega}(\cap_{k=1}^{\infty} \cup^{\infty}_{i=k} G_{i})\\
		&\leqslant& \lim_{k\rightarrow \infty}C_{q(.),p(.,.)}^{s,\Omega}(\cup_{i=k}^{\infty}G_{i})\\
		&\leqslant& \lim_{k\rightarrow \infty} \dfrac{1}{4^{k-1}}=0.
	\end{eqnarray*} 
	Thus 
	\begin{eqnarray*}
		C_{q(.),p(.,.)}^{s,\Omega}(\cap_{k=1}^{\infty} G_{k})=0.
	\end{eqnarray*}
	Consequently $\cap_{k=1}^{\infty} \cup^{\infty}_{i=k} G_{i}  $ is a $(s,q(.),p(.,.))$-relatively polar set. Moreover $u_{i}$ converges pointwise in $ \Omega\setminus \cap_{k=1}^{\infty} \cup^{\infty}_{i=k} G_{i}  $. Since $ \left|u_{i+1}(x)-u_{i}(x) \right|\leqslant \dfrac{1}{2^{i}}   $ in $ \Omega\setminus \cap_{k=1}^{\infty} \cup^{\infty}_{i=k} G_{i}  $ for all $i\geq k $, we have that $(u_{i})$ is a sequence of continuous functions on $\Omega$ which converges uniformly in $\Omega\setminus G_{k}$.          
\end{proof}
In the following we give sufficient conditions for the existence of a $(s,q(.),p(.,.))$-relatively quasicontinuous representative for functions in $\tilde{\mathcal{W}}^{s,q(.),p(.,.)}(\Omega)$.
\begin{definition}
	\item  A function $ u:\Omega\longrightarrow \mathbb{R}$ is said to be $(s,q(.),p(.,.))$-relatively quasicontinuous ($ (s,q(.),p(.,.)) $-r.q.c., for short) if for every $ \varepsilon>0 $, there exists a relatively open set $ O_{\varepsilon}\subset\Omega $ such that $C_{q(.),p(.,.)}^{s,\Omega}(O_{\varepsilon})<\varepsilon$ and $ u $ is continuous on $ \Omega\setminus O_{\varepsilon}$.    
\end{definition}
\begin{theorem}\label{quasicontinuous representative}
 Let $s\in(0,1), q\in \mathcal{P}(\Omega)$ and 
$p\in\mathcal{P}(\Omega\times\Omega)$. Then for every $ u\in\tilde{\mathcal{W}}^{s,q(.),p(.,.)}(\Omega) $, there exists a unique (up to a $p(.)$-relative polar set) $(s,q(.),p(.,.))$-r.q.c. function $\tilde{u}:\Omega\longrightarrow \mathbb{R}$ such that $\tilde{u}=u$ a.e. in $\Omega$. 
\end{theorem}
\begin{proof}
	Let $u\in\tilde{\mathcal{W}}^{s,q(.),p(.,.)}(\Omega)$. There exists a sequence $\mathcal{W}^{s,q(.),p(.,.)}(\Omega)\cap \mathscr{C}_{c}(\overline{\Omega})$ such that 
	$u_{i}\longrightarrow u$ in $\mathcal{W}^{s,q(.),p(.,.)}(\Omega)$. By Theorem \ref{subsequence which converges $p(.)$-r.q.e.} there exists a subsequence which converges $(s,q(.),p(.,.))$-r.q.e. in $ \Omega $ and uniformly outside a $(s,q(.),p(.,.))$-relative polar set. Let $\tilde{u}\longrightarrow u_{i}$ be the point-wise limit of 
	$(u_{i})$. By the uniform convergence we get that $\tilde{u}:\Omega\longrightarrow \mathbb{R}$ is $(s,q(.),p(.,.))$-r.q.c. $\tilde{u}=u$ a.e. in $\Omega$. For the uniqueness, we assume that there exists another $\tilde{v}$ $(s,q(.),p(.,.))$-r.q.c. in $\Omega$ such that $\tilde{v}=u$ a.e. in $\Omega$. Hence $\tilde{u}-\tilde{v}=0$ a.e. in $\Omega$ and $\tilde{u}-\tilde{v}=0$ is $(s,q(.),p(.,.))$-r.q.c. in $\Omega$.    
\end{proof}
\begin{corollary}\label{sous suite}
 Let $s\in(0,1), q\in \mathcal{P}(\Omega)$ and 
$p\in\mathcal{P}(\Omega\times\Omega)$. Let $(u_{i})$ be a sequence of $(s,q(.),p(.,.))$-r.q.c. functions in $\tilde{\mathcal{W}}^{s,q(.),p(.,.)}(\Omega)$ which converges to a $p(.)$-r.q.c. function $u\in\tilde{\mathcal{W}}^{s,q(.),p(.,.)}(\Omega)$. Then there exists a subsequence which converges
	$p(.)$-r.q.e. to $u$ on $\Omega$.
\end{corollary}
\begin{proof}
	Let $ (u_{i_{k}} )$ be a subsequence of $ (u_{i} )$ such that
	\begin{eqnarray*}
		\sum_{k=1}^{\infty} 2^{i_{k}}\left\| u_{i_{k}}-u \right\|_{\mathcal{W}^{s,q(.),p(.,.)}(\Omega)}\leq 1.
	\end{eqnarray*} 
	Put
	\begin{eqnarray*}
		P:=\cap_{j=1}^{\infty} \cup^{\infty}_{k=j} G_{k}, 
	\end{eqnarray*}
	where 
	\begin{eqnarray*}
		G_{k}:=\left\lbrace x\in\Omega: \ \left|u_{i_{k}}(x)-u(x) \right|>\dfrac{1}{2^{i_{k}}}.\right\rbrace. 
	\end{eqnarray*}
	There exists $ j_{0}\in \mathbb{N}$ such that  
	\begin{eqnarray*}
		\left|u_{i_{k}}(x)-u(x) \right|\leqslant \dfrac{1}{2^{i_{k}}},\ \forall \ k\geq j_{0}.  
	\end{eqnarray*}
	Hence $u_{i_{k}}(x)$ converges uniformly in $ \Omega\setminus \cup _{k=j_{0}}^{\infty}G_{k} $ and everywhere in $ \Omega\setminus P $. By the same way in the proof of Theorem \ref{subsequence which converges $p(.)$-r.q.e.} we get that
	\begin{eqnarray*}
		C_{q(.),p(.,.)}^{s,\Omega}(P)=0.
	\end{eqnarray*}
	Hence $P$ is $(s,q(.),p(.,.))$-relatively polar set and the proof is finished.
\end{proof}
The following result shows that two $(s,q(.),p(.,.))$-relatively quasicontinuous representative for functions in $\tilde{\mathcal{W}}^{s,q(.),p(.,.)}(\Omega)$ given by Theorem \ref{quasicontinuous representative} that agree almost everywhere coincide in fact $(s,q(.),p(.,.))$-relatively quasieverywhere. For the proof we refer to Kilpel\"{a}inen \cite{Kilpelainen} which is stated in metric measure spaces.
 \begin{theorem}\label{egalite quasipartout}
 Let $s\in(0,1), q\in \mathcal{P}(\Omega)$ and 
 $p\in\mathcal{P}(\Omega\times\Omega)$. Assume that $ u $ and $ v $ are $(s,q(.),p(.,.))$-relatively quasicontinuous representative for functions in $\tilde{\mathcal{W}}^{s,q(.),p(.,.)}(\Omega)$. If $ u=v $ a.e. in $ \Omega $, then $ u=v \ (s,q(.),p(.,.))$-r.q.e in $ \Omega $.   
 \end{theorem}  
\section{{\textbf{ Fractional Variable exposent Sobolev trace spaces }}}
We now turn our attention in the question of traces of Sobolev functions on the boundary of the set of definition. This problem is more delicate than the interior one since under some regularity assumptions on the variable exponents $ q $  and $ p$ it is possible to approximate a Sobolev function in the space $ \mathcal{W}^{s,q(.),p(.,.)}$ by smooth function; see \cite{BA1}, the same is not true up to the boundary.

Recall that if $\left( \mathcal{A},\left\Vert .\right\Vert _{\mathcal{A}
}\right) $ is a Banach space of measurable functions on $
\mathbb{R}^{n}$ and $E\subset\mathbb{R}^{n}$ is a measurable set of positive Lebesgue measure, then $\mathcal{A}_{\mid E}$ is the \emph{trace space} defined as 
\begin{equation*}
\mathcal{A}_{\mid E}:=\left\{ 
\begin{array}{ccc}
f: & E\rightarrow  & 
\mathbb{R}
\end{array}
;\text{ there exists }F\in \mathcal{A}\text{ such that }F_{\mid E}=f\text{
	a.e.}\right\} \text{.}
\end{equation*}
This space is equipped with the norm
\begin{equation*}
\left\Vert f\right\Vert _{\mathcal{A}_{\mid E}}=\inf \left\{ \left\Vert
F\right\Vert _{\mathcal{A}}:F\in \mathcal{A}\text{ , }F_{\mid E}=f\text{ a.e.
}\right\} \text{.}
\end{equation*}
Denoting the trace operator by $\mathcal{T}F=F_{\mid E}$. If $ \Omega$ is a smooth open set of $ \mathbb{R}^{n} $, then the characterization of traces is well known. The theorems of Baalal-Berghout \cite[Theorem 2.1 and Theorem 3.1 ] {BA}, states that if $ \Omega\subset \mathbb{R}^{n}$ is a Lipschitz domain, and under some assumptions on the variable exponents $ q $ and $ p $, then there exist trace and extension operators
\begin{equation}\label{trace operator}
T: \mathcal{W}^{s,q(.),p(.,.)}(\mathbb{R}^{n})\longrightarrow  \mathcal{W}^{s,q(.),p(.,.)}(\Omega),
\end{equation}
\begin{equation}\label{extension operator}
E:  \mathcal{W}^{s,q(.),p(.,.)}(\Omega)\longrightarrow  \mathcal{W}^{s,q(.),p(.,.)}(\mathbb{R}^{n}).
\end{equation}
Moreover, we have the following characterization: 
\begin{equation*}
\mathcal{W}^{s,q(.),p(.,.)}(\mathbb{R}^{n})=\ker T \oplus 
E(\mathcal{W}^{s,q(.),p(.,.)}(\Omega)).
\end{equation*}
Hence operators \eqref{trace operator} and \eqref{extension operator} characterizes traces on $\Omega$ of functions in $ \mathcal{W}^{s,q(.),p(.,.)}(\mathbb{R}^{n})$. For further results on variable exponent Sobolev trace spaces, see  \cite{Del pezzo,Diening and Hasto,Diening and Hasto 1,Edmunds, Edmunds 1,Harjulehto,Hasto3}; se also \cite{Kilpelainen1}.

In this section we study different definitions of fractional variable exponent Sobolev zero trace spaces in a proper open subset $ \Omega $ of $ \mathbb{R}^{n}$.
First, note that if $u\in \mathcal{W}^{s,q(.),p(.,.)}(\Omega)$ with compact support in $\Omega$. Then $ u$ is vanish on $\partial\Omega$. Indeed, let $u\in \mathcal{W}^{s,q(.),p(.,.)}(\Omega)$ and let $\psi\in \mathscr{C}^{\infty}_{0}(\Omega)$ be such that $\psi=1$ on the support of $u$. If a sequence $\psi_{j}$ converges to $u$ in $\mathcal{W}^{s,q(.),p(.,.)}(\Omega)$, then $\psi \psi_{j}$ converges to $\psi u=u$ in $\mathcal{W}^{s,q(.),p(.,.)}(\Omega)$.

\begin{definition}
	Let $s\in(0,1), q\in \mathcal{P}(\Omega)$ and 
$p\in\mathcal{P}(\Omega\times\Omega)$. The fractional Sobolev space $\mathcal{W}^{s,q(.),p(.,.)}_{0}(\Omega)$ is the closure of the set of $\mathcal{W}^{s,q(.),p(.,.)}(\Omega)$-functions with compact support, i.e.
\begin{equation*}
\left\lbrace u\in \mathcal{W}^{s,q(.),p(.,.)}(\Omega):u=u\chi_{K} \ \text{for a compact} \ K\subset \Omega  \right\rbrace. 
\end{equation*}
in $\mathcal{W}^{s,q(.),p(.,.)}(\Omega)$. 
\end{definition}
Let us looking $\mathcal{W}^{s,q(.),p(.,.)}_{0}(\Omega)$ as an ideal of $\mathcal{W}^{s,q(.),p(.,.)}(\Omega)$. The closed lattice ideals of the Sobolev spaces $\mathcal{W}^{s,q(.),p(.,.)}(\Omega)$ are those subspaces which consist of all functions which vanish on a prescribed set. To be precise, we have the following result:
\begin{theorem}\label{closed ideal}
	Let $s\in(0,1), q\in \mathcal{P}(\Omega)$ and 
	$p\in\mathcal{P}(\Omega\times\Omega)$. Then the space $\mathcal{W}^{s,q(.),p(.,.)}_{0}(\Omega)$ is a closed ideal in $\mathcal{W}^{s,q(.),p(.,.)}(\Omega)$.
\end{theorem}
\begin{proof}
	It is clear that $\mathcal{W}^{s,q(.),p(.,.)}_{0}(\Omega)$ is a closed Banach subspace of $\mathcal{W}^{s,q(.),p(.,.)}(\Omega)$. Let $u\in \mathcal{W}^{s,q(.),p(.,.)}_{0}(\Omega),\ v\in \mathcal{W}^{s,q(.),p(.,.)}(\Omega)$, $0 \leqslant\left|v \right| \leqslant \left|u \right|$ a.e. Let 
	$\varphi_{n}\in \mathscr{C}^{\infty}_{0}(\Omega)$ such that $\varphi_{n}$ converges to $u$ in $\mathcal{W}^{s,q(.),p(.,.)}_{0}(\Omega)$. Then $ v_{n}:=\min\left\lbrace v, \varphi_{n}, \right\rbrace $ has compact support and belongs to $\mathcal{W}^{s,q(.),p(.,.)}(\Omega)$. Thus $v_{n}\in \mathcal{W}^{s,q(.),p(.,.)}_{0}(\Omega)$. Moreover, since $ v_{n}\longrightarrow \min\left\lbrace v,u \right\rbrace =v$ in $\mathcal{W}^{s,q(.),p(.,.)}(\Omega)$, we have that $v\in \mathcal{W}^{s,q(.),p(.,.)}_{0}(\Omega)$.
\end{proof}  
We denote by $ \mathcal{{H}}^{s,q(.),p(.,.)}_{0}(\Omega) $ the closure of $\mathscr{C}_{0}^{\infty}(\Omega)$ in 
$\mathcal{W}^{s,q(.),p(.,.)}(\Omega)$. Recall that $ \mathcal{\tilde{{W}}}^{s,q(.),p(.,.)}(\Omega) $ is the closure of $\mathcal{W}^{s,q(.),p(.,.)}(\Omega)\cap \mathscr{C}_{c}(\overline{\Omega})$ in the space $ \mathcal{{W}}^{s,q(.),p(.,.)}(\Omega)$. It easy to show that $ \mathcal{\tilde{{W}}}^{s,q(.),p(.,.)}(\Omega) $ is a proper closed subspace of $ \mathcal{{W}}^{s,q(.),p(.,.)}(\Omega)$, and hence it is a separable, reflexive and uniformly convex Banach space. 

By definition, $ \mathcal{{H}}^{s,q(.),p(.,.)}_{0}(\Omega) $ is the smaller closed subspace of $ \mathcal{W}^{s,q(.),p(.,.)}(\Omega) $ containing $\mathscr{C}_{0}^{\infty}(\Omega)$. Moreover, $ \mathcal{\tilde{{W}}}^{s,q(.),p(.,.)}(\Omega) $ contains $ \mathcal{{H}}^{s,q(.),p(.,.)}_{0}(\Omega) $ as a closed subspace. Hence $ \mathcal{{H}}^{s,q(.),p(.,.)}_{0}(\Omega) $ is a Banach space and we have the following inclusion:
\begin{equation*}
\mathcal{{H}}^{s,q(.),p(.,.)}_{0}(\Omega)\subset \mathcal{{W}}^{s,q(.),p(.,.)}_{0}(\Omega)\subset \mathcal{{W}}^{s,q(.),p(.,.)}(\Omega).
\end{equation*}   

 The following theorem shows that the spaces $\mathcal{{H}}^{s,q(.),p(.,.)}_{0}(\Omega)$, $ \mathcal{\tilde{{W}}}^{s,q(.),p(.,.)}(\Omega) $ and $ \mathcal{W}^{s,q(.),p(.,.)}_{0}(\Omega) $ may coincide if smooth functions are dense in the fractional Sobolev space $ \mathcal{W}^{s,q(.),p(.,.)}(\Omega)$.
 \begin{theorem}\label{density condition}
 	Let $s\in(0,1), q\in \mathcal{P}(\Omega)$ and 
 	$p\in\mathcal{P}(\Omega\times\Omega)$. Assume that $ \mathcal{W}^{s,q(.),p(.,.)}(\Omega)\cap \mathscr{C}^{\infty}(\overline{\Omega})$ is dense in 
 	$ \mathcal{W}^{s,q(.),p(.,.)}(\Omega)$. 
 	Then
 	\begin{equation*}
 	\mathcal{{H}}^{s,q(.),p(.,.)}_{0}(\Omega)= \mathcal{W}^{s,q(.),p(.,.)}_{0}(\Omega)=\mathcal{\tilde{{W}}}^{s,q(.),p(.,.)}(\Omega).
 	\end{equation*}  
 \end{theorem}
 \begin{proof}
 	First, it is clear that if $ \mathcal{W}^{s,q(.),p(.,.)}(\Omega)\cap \mathscr{C}^{\infty}(\overline{\Omega})$ is dense in 
 	$ \mathcal{W}^{s,q(.),p(.,.)}(\Omega)$ then $ \mathcal{W}^{s,q(.),p(.,.)}_{0}(\Omega)=\mathcal{\tilde{{W}}}^{s,q(.),p(.,.)}(\Omega)$. Next, we show that 
 	\begin{equation*}
 	 \mathcal{{H}}^{s,q(.),p(.,.)}_{0}(\Omega)= \mathcal{W}^{s,q(.),p(.,.)}_{0}(\Omega).   
 	\end{equation*}
 	Clearly $\mathcal{{H}}^{s,q(.),p(.,.)}_{0}(\Omega)\subset \mathcal{W}^{s,q(.),p(.,.)}_{0}(\Omega)$. To show the other inclusion, let $ u\in \mathcal{{W}}^{s,q(.),p(.,.)}(\Omega) $ and $ K $ be a compact subset of $\Omega$  such that $u=u\chi_{K}$ almost everywhere. Let $ \varphi \in \mathscr{C}^{\infty}_{0}(\Omega)$ be such that $ 0\leqslant\varphi\leqslant1 $ and $ \varphi=1 $ in $ K $. There exists a sequence $ (u_{i})\subset \mathcal{W}^{s,q(.),p(.,.)}(\Omega)\cap \mathscr{C}^{\infty}(\overline{\Omega})$ converging to $ u$ in 
 	$\mathcal{W}^{s,q(.),p(.,.)}_{0}(\Omega)$ and 
 	\begin{eqnarray*}
 	\left\|u-\varphi u_{i} \right\|_{\mathcal{W}^{s,q(.),p(.,.)}(\Omega)}&\leq&\left\|u- u_{i} \right\|_{\mathcal{W}^{s,q(.),p(.,.)}(\Omega)}+\left\|u_{i}-\varphi u_{i} \right\|_{\mathcal{W}^{s,q(.),p(.,.)}(\Omega)}\\
 	 &\leq& \left\|u- u_{i} \right\|_{\mathcal{W}^{s,q(.),p(.,.)}(\Omega)}+\left\|u_{i}-\varphi u_{i} \right\|_{\mathcal{W}^{s,q(.),p(.,.)}(K)}\\
 	 &+&\left\|u_{i}-\varphi u_{i} \right\|_{\mathcal{W}^{s,q(.),p(.,.)}(\Omega\setminus K)}. 	
 	\end{eqnarray*} 
 Since $u_{i}$ converge to $ u $ as $ i $ tends to infinity and $ \varphi=1 $ in $ K $. We get that
 \begin{eqnarray*}
 	\left\|u-\varphi u_{i} \right\|_{\mathcal{W}^{s,q(.),p(.,.)}(\Omega)}&\leq&\left\|u- u_{i} \right\|_{\mathcal{W}^{s,q(.),p(.,.)}(\Omega)}+\left\|u_{i}-\varphi u_{i} \right\|_{\mathcal{W}^{s,q(.),p(.,.)}(\Omega)}\\
 	&\leq& \left\|u- u_{i} \right\|_{\mathcal{W}^{s,q(.),p(.,.)}(\Omega)}+\left\|u_{i}-\varphi u_{i} \right\|_{\mathcal{W}^{s,q(.),p(.,.)}(K)}\\
 	&+&\left\|u_{i}-\varphi u_{i} \right\|_{\mathcal{W}^{s,q(.),p(.,.)}(\Omega\setminus K)}. 	
 \end{eqnarray*} 
Now,by definition of $\mathcal{W}^{s,q(.),p(.,.)}_{0}(\Omega)$ for each function $u\in \mathcal{W}^{s,q(.),p(.,.)}_{0}(\Omega)$, we find a sequence in $\mathscr{C}^{\infty}_{0}(\Omega)$ converging to $ u $. Hence $ 	\left\|u-\varphi u_{i} \right\|_{\mathcal{W}^{s,q(.),p(.,.)}(\Omega)} \longrightarrow 0$ as $ i $ tends to infinity. Therefore $ \varphi u_{i}$ converges to $ u $ in $\mathcal{W}^{s,q(.),p(.,.)}(\Omega)$ as $ i $ tends to infinity. Thus we obtain $ \mathcal{W}^{s,q(.),p(.,.)}_{0}(\Omega) \subset\mathcal{{H}}^{s,q(.),p(.,.)}_{0}(\Omega)$.     
\end{proof}  
In connection with the density problem for smooth function, we introduce the most important condition on the exponent in the study of variable exponent spaces, the well-known log-H\"{o}lder continuity condition introduced by Zhikov in \cite{Zhikov}. We say that a function $ q:\Omega \rightarrow \mathbb{R}$ is log-H\"{o}lder continuous on $\Omega$ if there exists $C>0$ such that    
\begin{eqnarray*}
	\left| q(x)-q(y) \right|\leqslant \frac{C}{-\log \left|x-y \right| },\  \left|x-y \right|\leqslant \frac{1}{2}. 
\end{eqnarray*}
In \cite{BA1} Baalal and Berghout generalize this log-H\"{o}lder continuity condition for variable exponent $p\in\mathcal{P}(\Omega\times\Omega)$. We say that a function $ p:\Omega\times\Omega \rightarrow \mathbb{R}$ satisfies condition (B-B) on $\Omega\times\Omega$ if there exists $C>0$ such that    
\begin{eqnarray*}
	\left| p(x,y)-p(x^{'},y^{'}) \right|\leqslant \frac{C}{-\log (\left|x-x^{'} \right|+\left|y-y^{'} \right| ) },\  \left|x-x^{'} \right|+\left|y-y^{'} \right| \leqslant \frac{1}{2}. 
\end{eqnarray*} 
We define the following class of variable exponents 
\begin{eqnarray*}
	\mathcal{P}^{log}(\Omega):=\left\lbrace q:\Omega \rightarrow \mathbb{R}: q\  \textnormal{is measurable and log-H\"{o}lder continuous }\right\rbrace 
\end{eqnarray*}
and 
\begin{eqnarray*}
	\mathcal{P}^{log}(\Omega\times\Omega):=\left\lbrace p:\Omega\times\Omega \rightarrow \mathbb{R}: p\  \textnormal{is measurable and satisfies condition (B-B)}\right\rbrace. 
\end{eqnarray*}
We say that $ \Omega\subset \mathbb{R}^{{n}} $ is a $ \mathcal{W}^{s,q(.),p(.,.)}$-extension domain if there exists a continuous linear extension operator 
\begin{eqnarray*}
	E: \mathcal{W}^{s,q(.),p(.,.)}( \Omega ) \longrightarrow \mathcal{W}^{s,q(.),p(.,.)}(\mathbb{R}^{n})
\end{eqnarray*}
such that $ E u|_{\varOmega}=u$ for each $ u\in \mathcal{W}^{s,q(.),p(.,.)}(\Omega)$. A typical example of $ \mathcal{W}^{s,q(.),p(.,.)}$-extension domain is domains with Lipschitz boundary; see \cite{BA}.

Combining the Theorem \ref{density condition} and \cite[Theorem 3.3] {BA1}, we obtain the following corollary.
\begin{corollary}\label{ p log condition density}
Let $s\in(0,1), q\in\mathcal{P}^{log}(\Omega)$ and 
$p\in\mathcal{P}^{log}(\Omega\times\Omega)$. Assume that $ \Omega $ is a $\mathcal{W}^{s,q(.),p(.,.)}(\Omega)$-extension domain. Then
\begin{equation*}
\mathcal{{H}}^{s,q(.),p(.,.)}_{0}(\Omega)= \mathcal{W}^{s,q(.),p(.,.)}_{0}(\Omega)=\mathcal{\tilde{{W}}}^{s,q(.),p(.,.)}(\Omega).
\end{equation*}    
\end{corollary}
Notice that by previous study the class of smooth functions either can be (Theorem \ref{density condition} and Corollary 
\ref{ p log condition density}) or does not have to be dense in $ \mathcal{W}^{s,q(.),p(.,.)}(\Omega) $ depending on the variable exponents $ q$ and $ p $. Hence the closure of $\mathscr{C}_{0}^{\infty}(\Omega)$ under the fractional Sobolev norm is not a best way to define fractional variable exponents Sobolev zero trace spaces in every case.

 In view of potential theory (Theorem \ref{quasicontinuous representative}) a Sobolev function $ u $ in $\tilde{\mathcal{W}}^{s,q(.),p(.,.)}(\Omega)$ has a distinguished representative which is defined up to a set of $ (s,q(.),p(.,.))$-relative capacity zero. Therefore it is possible to look at traces of Sobolev functions on the boundary of the set of definition. Moreover, if $ E \subset \mathbb{R}^{n} $ with $C_{q(.),p(.,.)}^{s,\Omega}(E)>0$, then the trace of $ u $ to $ E $ is the restriction to $ E $ of any $(s,q(.),p(.,.))$-relatively quasicontinuous representative of $ u $.
 
 In the sequel, we give a characterization of $\mathcal{{H}}^{s,q(.),p(.,.)}_{0}(\Omega)$
and a necessary and sufficient condition in term of the $ (s,q(.),p(.,.))$-relative capacity for the equality
$\mathcal{{H}}^{s,q(.),p(.,.)}_{0}(\Omega)=\mathcal{\tilde{{W}}}^{s,q(.),p(.,.)}(\Omega)$ to be true. Let us beginning by the following lemma.
\begin{lemma}\label{limite nulle sur la frontiere}
	Let $\Omega \subset \mathbb{R}^{n}$ be a bounded open set and $ u \in \tilde{\mathcal{W}}^{s,q(.),p(.,.)}(\Omega)$. Let $\tilde{u}$ be the $(s,q(.),p(.,.))$-relatively quasicontinuous representative of $ u $. Assume that $ \lim_{\Omega\ni x\longrightarrow z} \tilde{u}(x)=0$ for all $ z \in\partial\Omega$. Then $ u \in \mathcal{H}^{s,q(.),p(.,.)}_{0}(\Omega)$.
\end{lemma}
\begin{proof}
	Let $ u \in \tilde{\mathcal{W}}^{s,q(.),p(.,.)}(\Omega)$. Recalling that $ u=u^{+}+u^{-} $, then without lost of generality, we can assume that $ u $ is nonnegative. Let $\tilde{u}$ be the $(s,q(.),p(.,.))$-relatively quasicontinuous representative of $ u $ with $ \lim_{\Omega\ni x\longrightarrow z} \tilde{u}(x)=0$ for all $ z \in\partial\Omega$. Then 
	\begin{equation*}
	(\forall z \in\partial\Omega)(\exists \varepsilon >0): \ 0\leq  \tilde{u}(x)\leq \dfrac{1}{2^{i}},
	\end{equation*}
	for all $ x\in B(z,\varepsilon)\cap \Omega $ and $ i\in \mathbb{N}$.\\
	Since $\partial\Omega$ is compact, there exist $ z_{1},\ldots, z_{j}\in \partial\Omega $ such that 
	\begin{equation*}
	\partial\Omega\subset\bigcup_{i=1}^{j}B(z_{i},\varepsilon_{z_{i}}).
	\end{equation*}
	By \cite[Lemma 2 ]{BA3}, $ u_{i}:=(u-\dfrac{1}{2^{i}})^{+}\in \tilde{\mathcal{W}}^{s,q(.),p(.,.)}(\Omega)$. Moreover, $ u_{i}=0 $ outside $\overline{\Omega}\setminus \bigcup_{i=1}^{j}B(z_{i},\varepsilon_{z_{i}})$. Hence $u_{i} \in \mathcal{H}^{s,q(.),p(.,.)}_{0}(\Omega)$. Since $ u_{i} $ converges to $ u $ in ${\mathcal{W}}^{s,q(.),p(.,.)}(\Omega)$ as $ i $ tends to infinity, we deduce that $ u \in \mathcal{H}^{s,q(.),p(.,.)}_{0}(\Omega)$.
\end{proof}
Notice that for a measurable set $ A\subset \overline{\Omega}$, the set 
\begin{equation*}
\mathcal{K}_{A}:=\left\lbrace u\in \tilde{\mathcal{W}}^{s,q(.),p(.,.)}(\Omega):u\geq 1 \ \text{a.e. on} \ A   \right\rbrace 
\end{equation*}
is a non-empty, closed convex subset of the uniform convex Banach space $\tilde{\mathcal{W}}^{s,q(.),p(.,.)}(\Omega)$. Hence, we deduce that there exist a unique admissible function $ \mathscr{E}_{A}$ in $ \mathcal{K}_{A} $ such that
\begin{equation*}
C_{q(.),p(.,.)}^{s,\Omega}(A)=\rho_{q(.),p(.,.)}^{s,\Omega}(\mathscr{E}_{A}).
\end{equation*}
This fact motivate the following definition.
\begin{definition}
	Let $ A\subset \overline{\Omega}$ be a measurable set. A function $\mathscr{E}_{A}\in\tilde{\mathcal{W}}^{s,q(.),p(.,.)}(\Omega)$ is called a relative $(s,q(.),p(.,.))$-equilibrium potential of $ A $ if the following properties hold.
	\begin{enumerate}
		\item $ 0\leqslant \mathscr{E}_{A} \leqslant 1$ a.e. on $ \Omega $ ;
		\item $ \tilde{\mathscr{E}}_{A}=1 $ $ (s,q(.),p(.,.))$-r.q.e. on $ A $, where $\tilde{\mathscr{E}}_{A}$ denotes the $(s,q(.),p(.,.))$-relatively quasicontinuous representative of $\mathscr{E}_{A}$;
		\item 
		$C_{q(.),p(.,.)}^{s,\Omega}(A)=\rho_{q(.),p(.,.)}^{s,\Omega}( \mathscr{E}_{A})$. 
	\end{enumerate} 
\end{definition} 
\begin{theorem}\label{inclusion for bounded set}
	Let $\Omega\subset \mathbb{R}^{n}$ be a bounded open set, $s\in(0,1)$, $q\in \mathcal{P}(\Omega)$ and 
	$p\in\mathcal{P}(\Omega\times\Omega)$. We set 
	\begin{eqnarray*}
		\mathcal{W}(\Omega):=\left\lbrace u\in \tilde{\mathcal{W}}^{s,q(.),p(.,.)}(\Omega):\ \tilde{u}=0 \ (s,q(.),p(.,.))\textnormal{-r.q.e. on} \ \partial\Omega   \right\rbrace, 	
	\end{eqnarray*} 
	where $\tilde{u}$ denotes the $(s,q(.),p(.,.))$-relatively quasicontinuous representative of $u$. Then
	\begin{equation*}
	\mathcal{W}(\Omega)\cap L^{\infty}(\Omega)\subset \mathcal{H}^{s,q(.),p(.,.)}_{0}(\Omega).	
	\end{equation*}
\end{theorem} 
\begin{proof}
	Let $ u\in 	\mathcal{W}(\Omega)\cap L^{\infty}(\Omega)$. Without lost of generality we can assume that $ u $ be nonnegative. By definition of the space $\tilde{\mathcal{W}}^{s,q(.),p(.,.)}(\Omega)$, there exist a sequence $ (u_{i})\subset \mathcal{W}^{s,q(.),p(.,.)}(\Omega)\cap \mathscr{C}_{c}(\overline{\Omega})$ of nonnegative functions which converges to $u\in\tilde{\mathcal{W}}^{s,q(.),p(.,.)}(\Omega)$ as $ i $ tends to infinity. There is a subsequence of $ (u_{i}) $, denoted again by 
	$(u_{i})$ and an open set $G_{i}\subset \overline{\Omega}$ such that $C_{q(.),p(.,.)}^{s,\Omega}(G_{i})<\dfrac{1}{2^{i}}$ and $ u_{i}\longrightarrow \tilde{u}$ uniformly on $\overline{\Omega}\setminus G_{i}$, where $\tilde{u}$ denotes the $(s,q(.),p(.,.))$-relatively quasicontinuous representative of $u$. Therefore, there exists $ i_{0}\in\mathbb{N}$ such that 
	\begin{equation*}
	\left|u-u_{i_{0}} \right|\leqslant \dfrac{1}{2^{i}} \ \ \text{everywhere on} \ \overline{\Omega}\setminus G_{i}    
	\end{equation*}
	and 
	\begin{equation*}
	\left\| u-u_{i_{0}}\right\|_{\mathcal{W}^{s,q(.),p(.,.)}(\Omega)}\leqslant \dfrac{1}{2^{i}}. 
	\end{equation*}
	Let $ O_{i}\subset \overline{\Omega} $ be an open set such that $C_{q(.),p(.,.)}^{s,\Omega}(O_{i})<\dfrac{1}{2^{i}}$ and $\tilde{u}=0 $ everywhere on $ \partial\Omega\setminus O_{i}$. Hence 
	\begin{equation*}
	\left|u_{i_{0}} \right|\leqslant \dfrac{1}{2^{i}} \ \ \text{everywhere on} \ \partial\Omega\setminus (O_{i}\cup G_{i}).    
	\end{equation*}
	Let $\mathscr{E}_{O_{i}\cup G_{i}}\in \tilde{\mathcal{W}}^{s,q(.),p(.,.)}(\Omega) $ be the relative $(s,q(.),p(.,.))$-equilibrium potential of $ O_{i}\cup G_{i} $, then
	$\mathscr{E}_{O_{i}\cup G_{i}}$ is $(s,q(.),p(.,.))$-relatively quasicontinuous,  $\mathscr{E}_{O_{i}\cup G_{i}}=1 $ everywhere on $O_{i}\cup G_{i}$ and $ 0\leqslant \mathscr{E}_{O_{i}\cup G_{i}} \leqslant 1$ a.e. on 
	$ \overline{\Omega}$. We define the sequence $ (v_{i}) $ by 
	\begin{equation*}
	v_{i}:=w_{i}(1-\mathscr{E}_{O_{i}\cup G_{i}}), \ \text{where} \ w_{i}=\max(0,u_{i_{0}}-\dfrac{1}{2^{i}}).
	\end{equation*} 
	Next, we show that $v_{i}\in\tilde{\mathcal{W}}^{s,q(.),p(.,.)}(\Omega)$. By following the same outlines given in \cite[Lemma 2.3]{BA} and taking into account that $ w_{i}, \mathscr{E}_{O_{i}\cup G_{i}}\in \mathcal{W}^{s,q(.),p(.,.)}(\Omega)\cap L^{\infty}(\Omega)$, we get that $v_{i}\in \mathcal{W}^{s,q(.),p(.,.)}(\Omega)$. Moreover, it is easy to see that $ v_{i} $ can be approximated in $ \mathcal{W}^{s,q(.),p(.,.)}(\Omega) $ by functions in $\mathcal{W}^{s,q(.),p(.,.)}(\Omega)\cap \mathscr{C}_{c}(\overline{\Omega})$. Hence, $ v_{i}\in \tilde{\mathcal{W}}^{s,q(.),p(.,.)}(\Omega)$. Now, let $ z\in\partial\Omega$ and $ z_{j}\in\Omega $ be such that $z_{j}$ converge to $ z $ as $j$ tends to infinity. Next, we prove that 
	$ \lim_{\Omega\ni x\longrightarrow z} v_{i}(x)=0$ for all $ z \in\partial\Omega$. \\
	If $ z\in\partial\Omega\setminus(O_{i}\cup G_{i})$, then $  v_{i}(z_{j})=w_{i}(z_{j})(1-\mathscr{E}_{O_{i}\cup G_{i}}(z_{j}))$. Thus, 
	\begin{eqnarray*}
		0\leqslant v_{i}(z_{j})\leqslant w_{i}(z_{j})\longrightarrow w_{i}(z) \ \text{as} \ j\longrightarrow\infty 
	\end{eqnarray*}
	and
	\begin{eqnarray*}
		w_{i}(z)=\max(0,u_{i_{0}}(z)-\dfrac{1}{2^{i}})\longrightarrow 0 \ \text{as} \ i\longrightarrow\infty. 
	\end{eqnarray*}
	If $ z\in\partial\Omega\cap(O_{i}\cup G_{i})$, then there exist $n_{0}\in\mathbb{N}$ such that $z_{j}\in(O_{i}\cup G_{i})$ for all $ i\geq n_{0}$. Hence,
	\begin{eqnarray*}
		0\leqslant v_{i}(z_{j})\leqslant w_{i}(z_{j})(1-\mathscr{E}_{O_{i}\cup G_{i}}(z_{j}))=0 \ \text{for all  } \ i\geq n_{0}.
	\end{eqnarray*}
	It follows that $ \lim_{\Omega\ni x\longrightarrow z} v_{i}(x)=0$ for all $ z \in\partial\Omega$. By Lemma \ref{limite nulle sur la frontiere}, we obtain that $ v_{i}\in \mathcal{H}^{s,q(.),p(.,.)}_{0}(\Omega)$. To finish the proof, we claim that $(v_{i})$ converges to $ u$ in $\mathcal{H}^{s,q(.),p(.,.)}_{0}(\Omega)$. Indeed, we have that
	\begin{eqnarray*}
		\left\|u-v_{i} \right\|_{\mathcal{W}^{s,q(.),p(.,.)}(\Omega)}&\leqslant& \left\|u-u_{i_{0}} \right\|_{\mathcal{W}^{s,q(.),p(.,.)}(\Omega)}+ \left\|u_{i_{0}}-w_{i}\right\|_{\mathcal{W}^{s,q(.),p(.,.)}(\Omega)}\\
		&+&\left\|w_{i}-v_{i}\right\|_{\mathcal{W}^{s,q(.),p(.,.)}(\Omega)}\\
		&\leqslant& \left\|u-u_{i_{0}} \right\|_{\mathcal{W}^{s,q(.),p(.,.)}(\Omega)}+ \left\|u_{i_{0}}-w_{i}\right\|_{\mathcal{W}^{s,q(.),p(.,.)}(\Omega)}\\
		&+&\left\|\mathscr{E}_{O_{i}\cup G_{i}}w_{i}\right\|_{\mathcal{W}^{s,q(.),p(.,.)}(\Omega)}\\
		&\leqslant& \left\|u-u_{i_{0}} \right\|_{\mathcal{W}^{s,q(.),p(.,.)}(\Omega)}+ \left\|u_{i_{0}}-w_{i}\right\|_{\mathcal{W}^{s,q(.),p(.,.)}(\Omega)}\\
		&+&\left\|\mathscr{E}_{O_{i}\cup G_{i}}w_{i}\right\|_{\mathcal{W}^{s,q(.),p(.,.)}(O_{i}\cup G_{i})}+\left\|\mathscr{E}_{O_{i}\cup G_{i}}w_{i}\right\|_{\mathcal{W}^{s,q(.),p(.,.)}(\Omega \setminus O_{i}\cup G_{i})}\\
		&\leqslant& \dfrac{1}{2^{i}}+\dfrac{2\max(\left| \Omega\right|^{\frac{1}{q^{+}}},\left| \Omega\right|^{\frac{1}{q^{-}}} )}{2^{i}}+\dfrac{4\max(\left| \Omega\right|^{\frac{1}{q^{+}}},\left| \Omega\right|^{\frac{1}{q^{-}}} )}{2^{i}}\left\|u \right\|_{\infty}. 	
	\end{eqnarray*}
	Hence, $(v_{i})$ converges to $ u$ in $\mathcal{H}^{s,q(.),p(.,.)}_{0}(\Omega)$ as $ i $ tends to infinity.    
\end{proof}
The following characterizations of the space $ \mathcal{H}^{s,q(.),p(.,.)}_{0}(\Omega) $ suggest that we  have functions that are the fractional relative $(s,q(.),p(.,.))$-capacity zero on the boundary $ \partial\Omega$. More precisely, we can describe $\mathcal{H}^{s,q(.),p(.,.)}_{0}(\Omega)$ as a subspace of $\tilde{\mathcal{W}}^{s,q(.),p(.,.)}(\Omega)$ in the following way:  
\begin{theorem}\label{description of H^{s,p}_{0} by relative capacity}
	Let $\Omega\subset \mathbb{R}^{n}$ be an open set, $s\in(0,1)$, $q\in \mathcal{P}(\Omega)$ and 
	$p\in\mathcal{P}(\Omega\times\Omega)$. Then
	\begin{eqnarray*}
		\mathcal{H}^{s,q(.),p(.,.)}_{0}(\Omega)=\left\lbrace u\in \tilde{\mathcal{W}}^{s,q(.),p(.,.)}(\Omega):\ \tilde{u}=0 \ (s,q(.),p(.,.))\textnormal{-r.q.e. on} \ \partial\Omega   \right\rbrace, 	
	\end{eqnarray*} 
	where $\tilde{u}$ denotes the $(s,q(.),p(.,.))$-relatively quasicontinuous representative of $u$.
\end{theorem} 
\begin{proof}
We set 
\begin{eqnarray*}
	\mathcal{W}(\Omega):=\left\lbrace u\in \tilde{\mathcal{W}}^{s,q(.),p(.,.)}(\Omega):\ \tilde{u}=0 \ (s,q(.),p(.,.))\textnormal{-r.q.e. on} \ \partial\Omega   \right\rbrace, 	
\end{eqnarray*} 
where $\tilde{u}$ denotes the $(s,q(.),p(.,.))$-relatively quasicontinuous representative of $u$. Let $ u\in\mathcal{H}^{s,q(.),p(.,.)}_{0}(\Omega)$, then there exists a sequence $ (u_{i})\in\mathscr{C}^{\infty}_{0}(\Omega)$ such that $u_{i}$ converges to $ u $ in $\tilde{\mathcal{W}}^{s,q(.),p(.,.)}(\Omega)$ as $ i $ tends to infinity. By Corollary \ref{sous suite}, there exists a subsequence denotes again by $u_{i}$ which converges to $\tilde{u}$ $(s,q(.),p(.,.))$-r.q.e. on $\overline{\Omega}$. Hence, $\tilde{u}=0$ $(s,q(.),p(.,.))$-r.q.e. on $\partial\Omega$, that is, $ u\in\mathcal{W}(\Omega) $. Hence 
\begin{equation*}
\mathcal{H}^{s,q(.),p(.,.)}_{0}(\Omega)\subset\mathcal{W}(\Omega).
\end{equation*}
We claim the converse inclusion. First, assume that $ \Omega $ is bounded and $ u\in 	\mathcal{W}(\Omega)\cap L^{\infty}(\Omega)$, then by Theorem \ref{inclusion for bounded set}, we get that
	\begin{equation*}
\mathcal{W}(\Omega)\cap L^{\infty}(\Omega)\subset \mathcal{H}^{s,q(.),p(.,.)}_{0}(\Omega).	
\end{equation*}
Now, if $ u\in \mathcal{W}(\Omega)$ is unbounded, then, for $i\in\mathbb{N}$, $  u_{i}:=\max\left\lbrace \min(u,i),-i \right\rbrace \in\mathcal{W}(\Omega)\cap L^{\infty}(\Omega)$. Next, we claim that $(u_{i})$ converges to $ u $ in $ \mathcal{W}^{s,q(.),p(.,.)}(\Omega)$. Indeed, let 
$u\in\mathcal{W}^{s,q(.),p(.,.)}(\Omega)$ and for $i\in\mathbb{N}$, $  u_{i}:=\max\left\lbrace \min(u,i),-i \right\rbrace$. Then
\begin{eqnarray*}
	u_{i}(x)=
	\left\lbrace
	\begin{array}{rl}
		i & \text{ on} \ D_{1},\\
		u & \text{ on} \ D_{2},\\
		-i & \text{ on} \ D_{3},
	\end{array}
	\right.
\end{eqnarray*}
where 
\begin{equation*}
 D_{1}:=\left\lbrace x\in\Omega,u(x)\geq i \right\rbrace, D_{2}:=\left\lbrace x\in\Omega,-i<u(x)< i \right\rbrace \   
\end{equation*}
and
\begin{equation*}
 \ D_{3}:=\left\lbrace x\in\Omega,u(x)\leqslant -i \right\rbrace.  
\end{equation*}
We have that $ u_{i}\longrightarrow u $ a.e. in $\Omega$ and $ \left|u_{i}(x)\right|^{q(x)}\leqslant \left|u(x)\right|^{q(x)}\in L^{1}(\Omega)$. By the Lebesgue dominated convergence theorem, we get that $ u_{i}\in L^{q(.)}(\Omega)$ and converges to $ u $ in $L^{q(.)}(\Omega)$. Moreover,  
\begin{eqnarray*}
\int_{\Omega}\int_{\Omega}\dfrac{\left| (u_{i}-u)(x)-(u_{i}-u)(y)\right|^{p(x,y)} }{\left|x-y \right|^{n+sp(x,y)}} \ dx \ dy= \int_{D_{1}}\int_{D_{1}}\dfrac{\left|u(x)-u(y)\right|^{p(x,y)} }{\left|x-y \right|^{n+sp(x,y)}} \ dx \ dy\\
+ \int_{D_{1}}\int_{D_{2}}\dfrac{\left|i-u(x)\right|^{p(x,y)} }{\left|x-y \right|^{n+sp(x,y)}} \ dx \ dy
+ \int_{D_{1}}\int_{D_{3}}\dfrac{\left|2i+u(y)-u(x)\right|^{p(x,y)} }{\left|x-y \right|^{n+sp(x,y)}} \ dx \ dy\\
+ \int_{D_{2}}\int_{D_{1}}\dfrac{\left|u(y)-i\right|^{p(x,y)} }{\left|x-y \right|^{n+sp(x,y)}} \ dx \ dy
+ \int_{D_{2}}\int_{D_{3}}\dfrac{\left|i+u(y)\right|^{p(x,y)} }{\left|x-y \right|^{n+sp(x,y)}} \ dx \ dy\\
+ \int_{D_{3}}\int_{D_{1}}\dfrac{\left|-2i+u(y)-u(x)\right|^{p(x,y)} }{\left|x-y \right|^{n+sp(x,y)}} \ dx \ dy
+ \int_{D_{3}}\int_{D_{2}}\dfrac{\left|i+u(x)\right|^{p(x,y)} }{\left|x-y \right|^{n+sp(x,y)}} \ dx \ dy\\
+ \int_{D_{3}}\int_{D_{3}}\dfrac{\left|u(y)-u(x)\right|^{p(x,y)} }{\left|x-y \right|^{n+sp(x,y)}} \ dx \ dy.										
\end{eqnarray*}
 Hence $ u_{i}\in \mathcal{W}^{s,q(.),p(.,.)}(\Omega)$. Moreover, applying the Lebesgue dominated convergence theorem to each term in the right hand side of the last equality, we get that 
 \begin{equation*}
 \int_{\Omega}\int_{\Omega}\dfrac{\left| (u_{i}-u)(x)-(u_{i}-u)(y)\right|^{p(x,y)} }{\left|x-y \right|^{n+sp(x,y)}} \ dx \ dy\longrightarrow 0 \ \text{as} \ i\longrightarrow\infty. 
 \end{equation*}
 Hence $ u_{i} $ converges to $ u $ in $ \mathcal{W}^{s,q(.),p(.,.)}(\Omega)$ and $u_{i}\in \mathcal{H}^{s,q(.),p(.,.)}_{0}(\Omega)$. Now, by Theorem \ref{inclusion for bounded set}, we have that $u\in \mathcal{H}^{s,q(.),p(.,.)}_{0}(\Omega)$. Finally, if $ \Omega $ is unbounded then we choose $ \varphi\in \mathscr{C}^{\infty}_{0}(B(0,r))$ such that $ \varphi=1 $ on $B(0,r)$. For $ u\in \mathcal{W}(\Omega)$ and $ i\in \mathbb{N}^{*} $, we set $ u_{i}(x):=\varphi(i^{-1}x)u(x) $. Then
 \begin{equation*}
 u_{i}\in \mathcal{W}(\Omega\cap B(0,i))\subset \mathcal{H}^{s,q(.),p(.,.)}_{0}(\Omega\cap B(0,i))\subset\mathcal{H}^{s,q(.),p(.,.)}_{0}(\Omega).
 \end{equation*}
 Since $ u_{i}$ converges to $ u $ in $ \mathcal{W}^{s,q(.),p(.,.)}(\Omega)$ as $ i $ tends to infinity, we get that 
 \begin{equation*}
  \mathcal{W}\subset \mathcal{H}^{s,q(.),p(.,.)}_{0}(\Omega).
 \end{equation*}   
\end{proof}  
The following corollary gives a necessary and sufficient condition in term of the $(s,q(.),p(.,.))$-relative capacity for the equality $\tilde{\mathcal{W}}^{s,q(.),p(.,.)} _{0}(\Omega)=\tilde{\mathcal{W}}^{s,q(.),p(.,.)}(\Omega)$.
\begin{corollary}
	Let $\Omega\subset \mathbb{R}^{n}$ be open. Then the following assertions are equivalent.
	\begin{enumerate}
		\item $ C_{q(.),p(.,.)}^{s,\Omega}(\partial\Omega)=0$;
		\item $\mathcal{H}^{s,q(.),p(.,.)}_{0}(\Omega)=\tilde{\mathcal{W}}^{s,q(.),p(.,.)}(\Omega)$.
	\end{enumerate}
\end{corollary}
The following definition for another fractional variable exponent Sobolev zero trace spaces $ \tilde{\mathcal{W}}^{s,q(.),p(.,.)}_{0}(\Omega) $ uses the $(s,q(.),p(.,.))$-relatively quasicontinuous representative.   
\begin{definition}
	Let $s\in(0,1), q\in \mathcal{P}(\Omega)$ and 
	$p\in\mathcal{P}(\Omega\times\Omega)$. We denote
$ u \in \tilde{\mathcal{W}}^{s,q(.),p(.,.)}_{0}(\Omega)$ and say that $ u $ belongs to the fractional variable exponent Sobolev zero trace spaces $ \tilde{\mathcal{W}}^{s,q(.),p(.,.)}_{0}(\Omega) $ if there exists a $(s,q(.),p(.,.))$-relatively quasicontinuous function $ \tilde{u}\in\tilde{\mathcal{W}}^{s,q(.),p(.,.)}(\Omega)$ such
that $ \tilde{u}= u$ a.e. in $\Omega$  and $ \tilde{u}=0$ $(s,q(.),p(.,.))$-r.q.e. in $ \Omega^{c} $.
\end{definition}
The set $ \tilde{\mathcal{W}}^{s,q(.),p(.,.)}_{0}(\Omega)$ is endowed with the norm 
\begin{equation*}
\left\|u \right\|_{ \tilde{\mathcal{W}}^{s,q(.),p(.,.)}_{0}(\Omega)}=\left\|\tilde{u} \right\|_{\tilde{\mathcal{W}}^{s,q(.),p(.,.)}(\Omega)}. 	
\end{equation*}
Observe that, from the definition of the fractional relative $(s,q(.),p(.,.))$-capacity, we deduce that every measurable set $ E\subset \overline{\Omega}$ satisfies $ \left|E \right|\leqslant C_{q(.),p(.,.)}^{s,\Omega}(E)$. It follows that, $ C_{q(.),p(.,.)}^{s,\Omega}(E)=0 $ implies that $\left|E \right|=0$. Hence, the norm  does not depend on the choice of the $(s,q(.),p(.,.))$-relatively quasicontinuous representative.

In contrast to $\mathcal{H}^{s,q(.),p(.,.)}_{0}(\Omega) $, the space $ \tilde{\mathcal{W}}^{s,q(.),p(.,.)}_{0}(\Omega) $ has the following fundamental property:
\begin{proposition}\label{produit dans Sobolev}
Let $u\in \tilde{\mathcal{W}}^{s,q(.),p(.,.)}_{0}(\Omega)$ and $ v\in \tilde{\mathcal{W}}^{s,q(.),p(.,.)}(\Omega)$ are bounded functions. Assume that $ v $ is $ (s,q(.),p(.,.))$-relatively quasicontinuous. Then $ uv\in \tilde{\mathcal{W}}^{s,q(.),p(.,.)}_{0}(\Omega)$. 
\end{proposition} 
\begin{proof}
Let $u\in \tilde{\mathcal{W}}^{s,q(.),p(.,.)}_{0}(\Omega)$ and $ v\in \tilde{\mathcal{W}}^{s,q(.),p(.,.)}(\Omega)$ are bounded functions. We assume that $ v $ is $ (s,q(.),p(.,.))$-relatively quasicontinuous. Then $uv\in\tilde{\mathcal{W}}^{s,q(.),p(.,.)}(\Omega)$. Let $\tilde{u}\in \tilde{\mathcal{W}}^{s,q(.),p(.,.)}(\Omega)$ be the $(s,q(.),p(.,.))$-relatively quasicontinuous representative of $ u $. Then $\tilde{u}v$ is $ (s,q(.),p(.,.))$-relatively quasicontinuous in $ \Omega$. We set $A:=\left\lbrace x\in\Omega^{c};\tilde{u}(x)\neq 0  \right\rbrace$ and $B:=\left\lbrace x\in\Omega^{c};v(x)= \infty  \right\rbrace$. Then $ C_{q(.),p(.,.)}^{s,\Omega}(A)=0$ and $ C_{q(.),p(.,.)}^{s,\Omega}(B)=0$. By the strong subadditivity of the fractional relative $ (s,q(.),p(.,.))$-capacity (Proposition \ref{strong subbaditivity de capacity}), we get that $ C_{q(.),p(.,.)}^{s,\Omega}(A\cup B)=0$. Hence $\tilde{u}v=0 \ (s,q(.),p(.,.)) \ \text{-q.e. in} \ \Omega^{c}$. Now, since $\tilde{u}v=uv \ \text{a.e. in} \ \Omega$, we deduce that $ uv\in \tilde{\mathcal{W}}^{s,q(.),p(.,.)}_{0}(\Omega)$.
\end{proof}
The next theorem gives a functional analysis-type properties of Sobolev spaces $ \tilde{\mathcal{W}}^{s,q(.),p(.,.)}_{0}(\Omega)$.  
\begin{theorem}\label{Banach quasi continuous}
	Let $s\in(0,1), q\in \mathcal{P}(\Omega)$ and 
$p\in\mathcal{P}(\Omega\times\Omega)$. Then $ \tilde{\mathcal{W}}^{s,q(.),p(.,.)}_{0}(\Omega)$ is a separable, reflexive and uniformly convex Banach lattice space.
\end{theorem}
\begin{proof}
Let $ (u_{i}) $ be a Cauchy sequence in $ \tilde{\mathcal{W}}^{s,q(.),p(.,.)}_{0}(\Omega)$. Then for every $ i=1,2,\ldots$ there is a $(s,q(.),p(.,.))$-relatively quasicontinuous representative $ \tilde{u_{i}}\in\tilde{\mathcal{W}}^{s,q(.),p(.,.)}(\Omega)$ of $u$. Since $ \tilde{\mathcal{W}}^{s,q(.),p(.,.)}(\Omega)$ is a closed subspace of a Banach space $\mathcal{W}^{s,q(.),p(.,.)}(\Omega)$, there is 
$u\in \tilde{\mathcal{W}}^{s,q(.),p(.,.)}(\Omega)$ such that 
$\tilde{u_{i}}$ converges to $ u $ in $ \tilde{\mathcal{W}}^{s,q(.),p(.,.)}(\Omega)$ as $ i $ tends to infinity. According to Corollary \ref{sous suite}, $u$ is $ (s,q(.),p(.,.)) $-relatively quasicontinuous and there is a subsequence of $(\tilde{u_{i}})$ which converges to $ u $ $ (s,q(.),p(.,.)) $-r.q.e
in $\tilde{\mathcal{W}}^{s,q(.),p(.,.)}(\Omega)$ as $ i $ tends to infinity. Hence $ u=0 $ $(s,q(.),p(.,.))$-r.q.e. in $\Omega^{c}$. Therefore $ u\in\tilde{\mathcal{W}}^{s,q(.),p(.,.)}_{0}(\Omega)$ and the space $ \tilde{\mathcal{W}}^{s,q(.),p(.,.)}_{0}(\Omega)$ is complete.

To show that $ \tilde{\mathcal{W}}^{s,q(.),p(.,.)}_{0}(\Omega)$ is a lattice, it is suffices to prove the assertions for $ \max\left\lbrace u,v\right\rbrace  $ since $ \min\left\lbrace u,v\right\rbrace =-\max\left\lbrace -u,-v\right\rbrace$. Let $u,v\in \tilde{\mathcal{W}}^{s,q(.),p(.,.)}_{0}(\Omega)$, then there exists a $ (s,q(.),p(.,.)) $-relatively quasicontinuous functions $ \tilde{u}, \tilde{v}\in\tilde{\mathcal{W}}^{s,q(.),p(.,.)}(\Omega)$ of $ u $ and $ v $ respectively. It is clear that $ \max\left\lbrace \tilde{u},\tilde{v}\right\rbrace  $ and $\min\left\lbrace \tilde{u},\tilde{v}\right\rbrace$ are also $(s,q(.),p(.,.))$-relatively quasicontinuous functions. Moreover,   
\begin{eqnarray*}
	\left\|\max\left\lbrace u,v\right\rbrace\right\|_{ \tilde{\mathcal{W}}^{s,q(.),p(.,.)}_{0}(\Omega)}&\leqslant&\left\|u \right\|_{ \tilde{\mathcal{W}}^{s,q(.),p(.,.)}_{0}(\Omega)}+\left\|v \right\|_{ \tilde{\mathcal{W}}^{s,q(.),p(.,.)}_{0}(\Omega)}\\
	&\leqslant& \left\|\tilde{u} \right\|_{\tilde{\mathcal{W}}^{s,q(.),p(.,.)}(\Omega)}+\left\|\tilde{v} \right\|_{\tilde{\mathcal{W}}^{s,q(.),p(.,.)}(\Omega)}\\
	 &<&\infty.	
\end{eqnarray*}     
Thus it follows that $\max\left\lbrace u,v\right\rbrace\in \tilde{\mathcal{W}}^{s,q(.),p(.,.)}_{0}(\Omega)$.\\
Since $\mathcal{\tilde{{W}}}^{s,q(.),p(.,.)}_{0}(\Omega) $ is a proper closed subspace of $ \mathcal{\tilde{{W}}}^{s,q(.),p(.,.)}(\Omega)$, it is also a separable, reflexive and uniformly convex space.   
\end{proof}
Using the fact that $\mathscr{C}^{\infty}_{0}(\Omega)\subset \tilde{\mathcal{W}}^{s,q(.),p(.,.)}_{0}(\Omega)$ and combining the Theorem \ref{Banach quasi continuous} with  the definition of the space $ \tilde{\mathcal{W}}^{s,q(.),p(.,.)}_{0}(\Omega)$, we obtain the following corollary.
\begin{corollary}\label{Inclusion spaces}
Let $s\in(0,1), q\in \mathcal{P}(\Omega)$ and 
$p\in\mathcal{P}(\Omega\times\Omega)$. Then
\begin{equation*}
\mathcal{H}^{s,q(.),p(.,.)}_{0}(\Omega)\subset \mathcal{\tilde{{W}}}^{s,q(.),p(.,.)}_{0}(\Omega)\subset\mathcal{\tilde{{W}}}^{s,q(.),p(.,.)}(\Omega).	
\end{equation*} 
\end{corollary}
 \begin{theorem}
	Let $s\in(0,1), q\in \mathcal{P}(\Omega)$ and 
	$p\in\mathcal{P}(\Omega\times\Omega)$. Assume that $ \mathcal{W}^{s,q(.),p(.,.)}(\Omega)\cap \mathscr{C}^{\infty}(\overline{\Omega})$ is dense in 
	$ \mathcal{W}^{s,q(.),p(.,.)}(\Omega)$. 
	Then
	\begin{equation*}
		\mathcal{{H}}^{s,q(.),p(.,.)}_{0}(\Omega)=\mathcal{\tilde{{W}}}^{s,q(.),p(.,.)}_{0}(\Omega)= \mathcal{W}^{s,q(.),p(.,.)}_{0}(\Omega)=\mathcal{\tilde{{W}}}^{s,q(.),p(.,.)}(\Omega).
	\end{equation*}  
\end{theorem}
\begin{proof}
By corollary \ref{Inclusion spaces} and Theorem \ref{density condition} it suffices to show that
\begin{equation*}
 \mathcal{\tilde{{W}}}^{s,q(.),p(.,.)}_{0}(\Omega)\subset \mathcal{{H}}^{s,q(.),p(.,.)}_{0}(\Omega). 
\end{equation*}
Let $ u\in\mathcal{\tilde{{W}}}^{s,q(.),p(.,.)}_{0}(\Omega) $ and $  \tilde{u}$ its canonical representative. We need to show that there exist function $ \varphi_{i}\in\mathscr{C}^{\infty}_{0}(\Omega)$ which converges to $  \tilde{u}$. Without lost of generality, we can assume that $\tilde{u}$ is positive and bounded.

Let $ \varepsilon>0 $ and $ G\subset\Omega $ be an open set such that $\tilde{u}$ is continuous in $ G^{c} $ and 
$C^{s,\Omega}_{q(.),p(.,.)}(G)<\varepsilon$. Since $\tilde{u}$ is continuous in $ G^{c} $ and $\tilde{u}=0$ a.e. in $ \Omega^{c}$, it follows that $\tilde{u}=0$ a.e. in $ (\Omega \cup G)^{c}$. Let $ \varphi_{\varepsilon}\in\mathcal{W}^{s,q(.),p(.,.)}(\Omega)$ be such that $ 0\leqslant \varphi_{\varepsilon}\leqslant1$, $ \varphi_{\varepsilon}=1 $ in $ G $, $ \left\| \varphi_{\varepsilon}\right\|_{L^{p(.,.)}(\Omega)}<\varepsilon$ and $ \rho^{s,\Omega}_{q(.),p(.,.)}(\varphi_{\varepsilon})< \varepsilon$. For $ 0<\delta<1 $, we define $\tilde{u}_{\delta}(x):=\max\left\lbrace \tilde{u}(x)-\delta,0 \right\rbrace$. Then, there exists a bounded closed set $ F\subset G^{c}$ such that $\tilde{u}_{\delta}(x)=0$ a.e. in $(F \cup G)^{c}$. Hence $ (1-\varphi_{\varepsilon})\tilde{u}_{\delta}=0 $ in $ F^{c} $ and $\text{supp}\left\lbrace (1-\varphi_{\varepsilon})\tilde{u}_{\delta} \right\rbrace\subset F$. Thus this function has compact support in $\Omega$. Next, we claim that $ (1-\varphi_{\varepsilon})\tilde{u}_{\delta}\longrightarrow \tilde{u} $ as $ \varepsilon,\delta\longrightarrow 0$. Observe that
\begin{equation*}
\left\|\tilde{u}-(1-\varphi_{\varepsilon})\tilde{u}_{\delta} \right\|_{\mathcal{W}^{s,q(.),p(.,.)}(\Omega)}\leq \left\|\tilde{u}-\tilde{u}_{\delta} \right\|_{\mathcal{W}^{s,q(.),p(.,.)}(\Omega)}+\left\|\varphi_{\varepsilon}\tilde{u}_{\delta} \right\|_{\mathcal{W}^{s,q(.),p(.,.)}(\Omega)} 	
\end{equation*}
and
\begin{equation*}
\left\|\tilde{u}-\tilde{u}_{\delta}\right\|_{\mathcal{W}^{s,q(.),p(.,.)}(\Omega)}\leqslant \delta \left\|\chi_{\text{supp} (\tilde{u})}\right\|_{{L}^{q(.)}(\Omega)}+\left\| \psi_{\delta}(x,y) \right\|_{L^{p(.,.)}(\Omega\times\Omega)}, 
\end{equation*}
where 
\begin{equation*}
 \psi_{\delta}(x,y):=\chi_{\left\lbrace 0<\tilde{u}(x)<\delta\right\rbrace\times \left\lbrace 0<\tilde{u}(y)<\delta\right\rbrace} \dfrac{\left|\tilde{u}(x)-\tilde{u}(y)\right|}{\left|x-y \right|^{n+\frac{s}{p(x,y)}}}.
\end{equation*}
Hence $\left\|\tilde{u}-\tilde{u}_{\delta}\right\|_{\mathcal{W}^{s,q(.),p(.,.)}(\Omega)}\longrightarrow 0$ as $ \delta\longrightarrow 0$.\\  
Next, we show that $ \left\|\varphi_{\varepsilon}\tilde{u} \right\|_{\mathcal{W}^{s,q(.),p(.,.)}(\Omega)}\longrightarrow 0 $ as $\varepsilon\longrightarrow 0 $. From \cite [Proposition 3]{BA3}, it suffices to prove that $\rho^{s,\Omega}_{q(.),p(.,.)}(\varphi_{\varepsilon}\tilde{u}_{\delta})\longrightarrow 0$ as $ \varepsilon\longrightarrow 0 $. We have that  
\begin{eqnarray*}
	\rho^{s,\Omega}_{q(.),p(.,.)}(\varphi_{\varepsilon}\tilde{u}):=\int_{\Omega}\left| \varphi_{\varepsilon}(x)\tilde{u}(x)\right|^{q(x)}\ dx+\int_{\Omega}\int_{\Omega} \frac{\left|(\varphi_{\varepsilon}\tilde{u})(x)-(\varphi_{\varepsilon}\tilde{u})(y)\right|^{p(x,y)}}{\left|x-y\right|^{n+ sp(x,y)}}\ dx \ dy.	
\end{eqnarray*}
Since $\tilde{u}$ is bounded, we obtain that
\begin{eqnarray*}
\int_{\Omega}\left| \varphi_{\varepsilon}(x)\tilde{u}(x)\right|^{q(x)}\ dx&\leqslant& (\left\| \tilde{u}\right\|_{\infty}^{q^{+}}+\left\| \tilde{u}\right\|_{\infty}^{q^{-}})\int_{\Omega}\left| \varphi_{\varepsilon}(x)\right|^{q(x)}\ dx\\
&\leqslant& (\left\| \tilde{u}\right\|_{\infty}^{q^{+}}+\left\| \tilde{u}\right\|_{\infty}^{q^{-}}) 	\rho^{s,\Omega}_{q(.),p(.,.)}(\varphi_{\varepsilon})\\
&\leqslant& \varepsilon(\left\| \tilde{u}\right\|_{\infty}^{q^{+}}+\left\| \tilde{u}\right\|_{\infty}^{q^{-}}). 
\end{eqnarray*}
Hence $ \int_{\Omega}\left| \varphi_{\varepsilon}(x)\tilde{u}(x)\right|^{q(x)} \ dx \longrightarrow 0$ as $ \varepsilon\longrightarrow0$. On the other hand
\begin{eqnarray*}
\left|\varphi_{\varepsilon}(x)\tilde{u}(x)-\varphi_{\varepsilon}(y)\tilde{u}(y)\right|^{p(x,y)}&\leq& 2^{p^{+}-1}\left|\tilde{u}(x) \right|^{p(x,y)}\left|\varphi_{\varepsilon}(x)-\varphi_{\varepsilon}(y) \right|^{p(x,y)}\\
&+&
2^{p^{+}-1}\left|\varphi_{\varepsilon}(y) \right|^{p(x,y)}\left|\tilde{u}(x)-\tilde{u}(y) \right|^{p(x,y)}.   
\end{eqnarray*}
Hence
\begin{eqnarray*}
\int_{\Omega}\int_{\Omega} \frac{\left|(\varphi_{\varepsilon}\tilde{u})(x)-(\varphi_{\varepsilon}\tilde{u})(y)\right|^{p(x,y)}}{\left|x-y\right|^{n+ sp(x,y)}} \ dx \ dy&\leq& 2^{p^{+}-1}\int_{\Omega}\int_{\Omega}\dfrac{\left|\tilde{u}(x)\right|^{p(x,y)}\left|\varphi_{\varepsilon}(x)-\varphi_{\varepsilon}(y) \right|^{p(x,y)}}{\left|x-y\right|^{n+ sp(x,y)}} \ dx \ dy\\
&+&2^{p^{+}-1}\int_{\Omega}\int_{\Omega}\dfrac{\left|\varphi_{\varepsilon}(y)\right|^{p(x,y)}\left|\tilde{u})(x)-\tilde{u})(y) \right|^{p(x,y)}}{\left|x-y\right|^{n+ sp(x,y)}} \ dx \ dy\\
&\leqslant& (\left\| \tilde{u}\right\|_{\infty}^{p^{+}}+\left\| \tilde{u}\right\|_{\infty}^{p^{-}}) 	\rho^{s,\Omega}_{q(.),p(.,.)}(\varphi_{\varepsilon})\\
&+& 2^{p^{+}-1}\int_{\Omega}\int_{\Omega}\dfrac{\left|\varphi_{\varepsilon}(y)\right|^{p(x,y)}\left|\tilde{u})(x)-\tilde{u})(y) \right|^{p(x,y)}}{\left|x-y\right|^{n+ sp(x,y)}} \ dx \ dy.  
\end{eqnarray*}
Since $ \varphi_{\varepsilon}\longrightarrow 0$ in $ L^{p(.,.)}(\Omega\times\Omega) $, as $ \varepsilon\longrightarrow 0 $, we can choose a sequence $ \varphi_{i} $ which tends to $ 0 $ a.e. in $\Omega\times\Omega$. Then
\begin{equation*}
\int_{\Omega}\dfrac{\left|\varphi_{i}(y)\right|^{p(x,y)}\left|\tilde{u})(x)-\tilde{u})(y) \right|^{p(x,y)}}{\left|x-y\right|^{n+ sp(x,y)}} \ dx \ dy\longrightarrow 0
\end{equation*}
by the Lebesgue dominated convergence theorem with $\dfrac{\left|\tilde{u}(x)-\tilde{u}(y) \right|^{p(x,y)}}{\left|x-y\right|^{n+ sp(x,y)}}$ as a dominant. Therefore $\rho^{s,\Omega}_{q(.),p(.,.)}(\varphi_{\varepsilon}\tilde{u})\longrightarrow 0$, as $ \varepsilon\longrightarrow 0 $. Hence $ (1-\varphi_{\varepsilon})\tilde{u}_{\delta}\longrightarrow \tilde{u}$ as $ \varepsilon,\delta\longrightarrow 0 $. Now, let $ \psi_{i}\in \mathscr{C}^{\infty}(\Omega)$ which tend to $(1-\varphi_{\varepsilon})\tilde{u}_{\delta}$. Let $ \eta\in \mathscr{C}^{\infty}_{0}(\Omega)$ such that $ \eta=1 $ in $\text{supp}\left\lbrace (1-\varphi_{\varepsilon}) \tilde{u}_{\delta}\right\rbrace$. By construction, we have that $ \eta\psi_{i}\in \mathscr{C}^{\infty}_{0}(\Omega)$. We set $v:=(1-\varphi_{\varepsilon})\tilde{u}_{\delta}$, then
\begin{eqnarray*}
\rho^{s,\Omega}_{q(.),p(.,.)}(v-\eta\psi_{i})&=&\int_{\text{supp}\left\lbrace v(x)\right\rbrace }\left| v(x)-\psi_{i}(x)\right|^{q(x)} \ dx\\
&+&\int_{\text{supp}\left\lbrace v(x)\right\rbrace }\int_{\text{supp}\left\lbrace v(y)\right\rbrace } 
\dfrac{\left|v(x)-\psi_{i}(x)-v(y)+\psi_{i}(y) \right|^{p(x,y)}}{\left|x-y\right|^{n+ sp(x,y)}}  \ dx  \ dy\\
&+& \int_{\Omega\setminus\text{supp}\left\lbrace v(x)\right\rbrace }\left| \eta(x)\psi_{i}(x)\right|^{q(x)} \ dx \\
&+& \int_{\Omega\setminus\text{supp}\left\lbrace v(x)\right\rbrace } \int_{\Omega\setminus\text{supp}\left\lbrace v(y)\right\rbrace } \dfrac{\left| \eta(y)\psi_{i}(y)-\eta(x)\psi_{i}(x)\right|^{p(x,y)} }{\left|x-y \right|^{n+sp(x,y)}} \ dx  \ dy.  
\end{eqnarray*}
Since $ \psi_{i}\longrightarrow v$, then 
\begin{equation*}
\int_{\text{supp}\left\lbrace v(x)\right\rbrace }\left| v(x)-\psi_{i}(x)\right|^{q(x)} \ dx\longrightarrow 0
\end{equation*} and 
\begin{equation*}
 \int_{\text{supp}\left\lbrace v(x)\right\rbrace }\int_{\text{supp}\left\lbrace v(y)\right\rbrace } 
\dfrac{\left|v(x)-\psi_{i}(x)-v(y)+\psi_{i}(y) \right|^{p(x,y)}}{\left|x-y\right|^{n+ sp(x,y)}}  \ dx  \ dy\longrightarrow 0.   
\end{equation*}
Moreover,
\begin{eqnarray*}
\int_{\Omega\setminus\text{supp}\left\lbrace v(x)\right\rbrace }\left|\eta(x)\psi_{i}(x) \right|^{q(x)} \ dx\leqslant (\left\| \eta\right\|_{\infty}^{q^{+}}+\left\| \eta\right\|_{\infty}^{q^{-}})\int_{\Omega\setminus\text{supp}\left\lbrace v(x)\right\rbrace }\left|\psi_{i}(x) \right|^{q(x)} \ dx 
\end{eqnarray*}
and
\begin{eqnarray*}
\int_{\Omega\setminus\text{supp}\left\lbrace v(x)\right\rbrace } \int_{\Omega\setminus\text{supp}\left\lbrace v(y)\right\rbrace } \dfrac{\left| \eta(y)\psi_{i}(y)-\eta(x)\psi_{i}(x)\right|^{p(x,y)} }{\left|x-y \right|^{n+sp(x,y)}} \ dx  \ dy
\end{eqnarray*}
is less than
\begin{eqnarray*} 
 C(p^{+},p^{-},\left\| \eta\right\|_{\infty})\int_{\Omega\setminus\text{supp}\left\lbrace v(x)\right\rbrace } \int_{\Omega\setminus\text{supp}\left\lbrace v(y)\right\rbrace }\left( \dfrac{\left| \psi_{i}(x)-\psi_{i}(y)\right|^{p(x,y)} }{\left|x-y \right|^{n+sp(x,y)}} \ + \left| \psi_{i}(x)\right|^{p(x,y)}  \right) \ dx  \ dy.   
\end{eqnarray*} 
Since $\psi_{i}\longrightarrow v  $ and $ v=0 $ in $ \Omega\setminus\text{supp}\left\lbrace v\right\rbrace$, we get that
\begin{equation*}
\int_{\Omega\setminus\text{supp}\left\lbrace v(x)\right\rbrace }\left| \eta(x)\psi_{i}(x)\right|^{q(x)} \ dx\longrightarrow 0
\end{equation*}
and 
 \begin{equation*}
 \int_{\Omega\setminus\text{supp}\left\lbrace v(x)\right\rbrace } \int_{\Omega\setminus\text{supp}\left\lbrace v(y)\right\rbrace } \dfrac{\left| \eta(y)\psi_{i}(y)-\eta(x)\psi_{i}(x)\right|^{p(x,y)} }{\left|x-y \right|^{n+sp(x,y)}} \ dx  \ dy\longrightarrow 0.
 \end{equation*}
 Hence $ \rho^{s,\Omega}_{q(.),p(.,.)}(v-\eta\psi_{i})\longrightarrow 0$. Therefore, we have constructed a sequence $ (\eta\psi_{i})_{i}\in \mathscr{C}^{\infty}_{0}(\Omega)$ which approaches $v $. But $ v $ can be chosen arbitrarily close to $\tilde{u} $, and so we get a sequence of $ \mathscr{C}^{\infty}_{0}(\Omega) $ functions tending to $\tilde{u} $.     
\end{proof}
\begin{corollary}
	Let $s\in(0,1), q\in\mathcal{P}^{log}(\Omega)$ and 
	$p\in\mathcal{P}^{log}(\Omega\times\Omega)$. Assume that $ \Omega $ is a $\mathcal{W}^{s,q(.),p(.,.)}_{0}(\Omega)$-extension domain. Then
	\begin{equation*}
		\mathcal{{H}}^{s,q(.),p(.,.)}_{0}(\Omega)= \mathcal{W}^{s,q(.),p(.,.)}_{0}(\Omega)=\mathcal{\tilde{{W}}}^{s,q(.),p(.,.)}(\Omega)=\mathcal{\tilde{{W}}}^{s,q(.),p(.,.)}_{0}(\Omega).
	\end{equation*}    
\end{corollary}

 Next, we give a relative capacity criterium for removable subsets for $\tilde{\mathcal{W}}^{s,q(.),p(.,.)}_{0}(\Omega)$.
 \begin{definition}
 	Let $ \Omega\subset \mathbb{R}^{n}$ be an open set. We call a subset $ N $ of $ \Omega $ is removable for $\mathcal{\tilde{{W}}}^{s,q(.),p(.,.)}_{0}(\Omega)$ if 
 	\begin{equation*}
 		\mathcal{\tilde{{W}}}^{s,q(.),p(.,.)}_{0}(\Omega)=\mathcal{\tilde{{W}}}^{s,q(.),p(.,.)}_{0}(\Omega\setminus N).	
 	\end{equation*}  
 \end{definition}
\begin{theorem}
Let $ N $ be a subset of $\Omega$. Then the following assertions are equivalent.
\begin{enumerate}
	\item $ C_{q(.),p(.,.)}^{s,\Omega}(N\cap\Omega)=0$;
	\item $\tilde{\mathcal{W}}^{s,q(.),p(.,.)}_{0}(\Omega)=\tilde{\mathcal{W}}^{s,q(.),p(.,.)}_{0}(\Omega\setminus N)$.
\end{enumerate}
\end{theorem}
\begin{proof}
$(1)\Longrightarrow (2)$: We assume that $ C_{q(.),p(.,.)}^{s,\Omega}(N\cap\Omega)=0$. First, it is clear that $\tilde{\mathcal{W}}^{s,q(.),p(.,.)}_{0}(\Omega\setminus N)\subset \tilde{\mathcal{W}}^{s,q(.),p(.,.)}_{0}(\Omega)$. To prove the converse inclusion, let $u\in\tilde{\mathcal{W}}^{s,q(.),p(.,.)}_{0}(\Omega)$ and $\tilde{u}\in\tilde{\mathcal{W}}^{s,q(.),p(.,.)}(\Omega)$ be it is relative $(s,q(.),p(.,.))$-quasicontinuous representative. Hence $  \tilde{u}=u$ a.e. in $\Omega\setminus N$ and $\tilde{u}=0 \ (s,q(.),p(.,.)$-r.q.e in $(\Omega\setminus N)^{c}$. Hence $ u $ restricted to $ \Omega\setminus N $ belongs to $\tilde{\mathcal{W}}^{s,q(.),p(.,.)}_{0}(\Omega\setminus N)$ and
\begin{equation*}
	\left\| u_{|(\Omega\setminus N)} \right\|_{ \tilde{\mathcal{W}}^{s,q(.),p(.,.)}_{0}(\Omega\setminus N)}=\left\|\tilde{u} \right\|_{\tilde{\mathcal{W}}^{s,q(.),p(.,.)}(\Omega)}. 	
\end{equation*}
Thus $\tilde{\mathcal{W}}^{s,q(.),p(.,.)}_{0}(\Omega)\subset\tilde{\mathcal{W}}^{s,q(.),p(.,.)}_{0}(\Omega\setminus N)$.\\
$(2)\Longrightarrow (1)$: Let $ N $ be a subset of $\Omega$. We assume that 
\begin{equation*}
	\tilde{\mathcal{W}}^{s,q(.),p(.,.)}_{0}(\Omega)=\tilde{\mathcal{W}}^{s,q(.),p(.,.)}_{0}(\Omega\setminus N). 
\end{equation*}
Let $ x_{0}\in\Omega $ and write
\begin{equation*}
	\Omega_{i}:=B(x_{0},i)\cap\left\lbrace x\in\Omega:dist(x,\Omega^{c})>\dfrac{1}{i} \right\rbrace, i=1,2,\ldots 
\end{equation*}
We define $u_{i}:\Omega\longrightarrow \mathbb{R}$ by $u_{i}:=\max\left\lbrace 1-dist(x,N\cap\Omega_{i}),0\right\rbrace$, $i=1,2,\ldots$. Then $ u_{i}\in\tilde{\mathcal{W}}^{s,q(.),p(.,.)}(\Omega)$ is continuous, $ 0\leqslant u_{i}\leqslant 1 $ and $ u_{i}=1 $ in $N\cap\Omega_{i}$. We define $v_{i}:\Omega_{i}\longrightarrow \mathbb{R}$ by 
$v_{i}:=dist(x,\Omega_{i}^{c})$, $i=1,2,\ldots$ Then $ v_{i} $ is $(s,q(.),p(.,.))$-relatively quasicontinuous and $ v_{i}\in \tilde{\mathcal{W}}^{s,q(.),p(.,.)}_{0}(\Omega_{i})\subset\tilde{\mathcal{W}}^{s,q(.),p(.,.)}_{0}(\Omega)$. By Proposition \ref{produit dans Sobolev}, we have that $ u_{i}v_{i}\in\tilde{\mathcal{W}}^{s,q(.),p(.,.)}_{0}(\Omega)=\tilde{\mathcal{W}}^{s,q(.),p(.,.)}_{0}(
\Omega\setminus N)$, $i=1,2,\ldots$ We fix $ i $. If $ w $ is such that $(s,q(.),p(.,.))$-relatively quasicontinuous function that $w=u_{i}v_{i}$ a.e. in $\Omega\setminus N$, then $w=u_{i}v_{i} \ \text{a.e in} \ \Omega$ since $ \left|N \right|=0$. Now, by Theorem \ref{egalite quasipartout} we get that $w=u_{i}v_{i} \ (s,q(.),p(.,.)) \ \text{-r.q.e in} \ \Omega$. In particular, $w=u_{i}v_{i}>0 \ (s,q(.),p(.,.)) \ \text{-r.q.e in} \ \Omega_{i}\cap N$. On the other hand, since $ u_{i}v_{i}\in \tilde{\mathcal{W}}^{s,q(.),p(.,.)}_{0}(\Omega\setminus N)$, we can define $ w=u_{i}v_{i}=0 \ (s,q(.),p(.,.)) \ \text{-r.q.e in} \ (\Omega\setminus N)^{c}$. In particular, we have $ w=u_{i}v_{i}=0 \ (s,q(.),p(.,.)) \ \text{-r.q.e in} \ N\setminus\Omega_{i}$. This is possible only if $C_{q(.),p(.,.)}^{s,\Omega}(N\setminus\Omega_{i})=0$ for $i=1,2,\ldots$ By the proprieties of the fractional relative $(s,q(.),p(.,.))$-capacity, we get that
\begin{eqnarray*}
 C_{q(.),p(.,.)}^{s,\Omega}(N)&\leqslant& C_{q(.),p(.,.)}^{s,\Omega}\left( \bigcup_{i=1}^{\infty} (N\cap \Omega _{i})\right)\\
 &\leqslant& \sum _{i=1}^{\infty}C_{q(.),p(.,.)}^{s,\Omega}( N\cap \Omega _{i})=0.
\end{eqnarray*}
This complete the proof.    
\end{proof}
\bibliographystyle{amsplain}

\end{document}